\documentclass[reqno]{amsart}
\usepackage[left=2.5cm,top=2.5cm,right=2.5cm,bottom=2.5cm]{geometry}
\usepackage{enumerate}
\usepackage{csquotes}

\usepackage{amssymb}
\usepackage{amsmath}
\usepackage{amsfonts}
\usepackage{amssymb}
\usepackage{amsthm}

\usepackage[ruled,vlined]{algorithm2e}

\newtheorem{thm}{Theorem}
\newtheorem{lem}[thm]{Lemma}
\newtheorem{prop}[thm]{Proposition}
\newtheorem{claim}[thm]{Claim}

\newtheorem{definition}[thm]{Definition}

\numberwithin{thm}{section}

\newcommand{\Bi}{\mathrm{Bi}}
\newcommand{\Po}{\mathrm{Po}}
\newcommand{\pr}{\mathbb{P}}
\renewcommand{\Pr}{\mathbb{P}}

\newcommand{\cH}{\mathcal{H}}
\newcommand{\cE}{\mathcal{E}}
\newcommand{\cC}{\mathcal{C}}
\newcommand{\cQ}{\mathcal{Q}}
\newcommand{\cA}{\mathcal{A}}
\newcommand{\cR}{\mathcal{R}}
\newcommand{\cD}{\mathcal{D}}
\newcommand{\cB}{\mathcal{B}}

\newcommand{\ii}{\mathbf{i}}
\newcommand{\UU}{\mathbf{U}}

\newcommand{\NN}{\mathbb{N}}

\newcommand{\cond}{\; \middle\vert \;}

\newcommand{\eps}{\varepsilon}

\newcommand{\Erdos}{Erd\H{o}s}
\newcommand{\Renyi}{R\'enyi}

\newcommand{\sizevl}{ n-n^{1-\delta}}
\newcommand{\error}{\eps^*}

\newcommand{\pc}{M}

\begin{document}

\title{The sharp threshold for jigsaw percolation in random graphs}

\author[Oliver Cooley, Tobias Kapetanopoulos, Tam\'as Makai] {Oliver Cooley$^{*}$, Tobias Kapetanopoulos$^{**}$, Tam\'as Makai$^{***}$}

\thanks{$^{*}$Supported by Austrian Science Fund (FWF) \& German Research Foundation (DFG): I3747.\\
		\indent $^{**}$Supported by Stiftung Polytechnische Gesellschaft PhD grant.\\
		\indent $^{***}$Supported by Austrian Science Fund (FWF): P26826 and EPSRC: EP/N004221/1.}

\address{Oliver Cooley, {\tt cooley@math.tugraz.at}, Graz University of Technology, Institute of Discrete Mathematics, Steyrergasse 30, 8010 Graz, Austria.}
\address{Tobias Kapetanopoulos, {\tt kapetano@math.uni-frankfurt.de}, Goethe University, Mathematics Institute, 10 Robert Mayer St, Frankfurt 60325, Germany.}
\address{Tam\'as Makai, {\tt t.makai@qmul.ac.uk}, School of Mathematical Sciences, Queen Mary University of London, Mile End Road, London E1 4NS.}

\begin{abstract}
We analyse the jigsaw percolation process, which may be seen as a measure of whether two graphs on the same vertex set are
``jointly connected''. Bollob\'as, Riordan, Slivken and Smith proved that when the two graphs are independent binomial random graphs,
whether the jigsaw process percolates undergoes a phase transition when the product of the two probabilities
is $\Theta\left( \frac{1}{n\ln n} \right)$. We show that this threshold is sharp, and that it lies at $\frac{1}{4n\ln n}$.\\
\\
\begin{tabular}{l}
\end{tabular}\\
\begin{tabular}{rl}
2010 Mathematics Subject Classification: & Primary 05C80\\
& Secondary 60K35, 60C05.
\end{tabular}
\end{abstract}

\maketitle

\section{Introduction}

\subsection{Motivation and history}

The jigsaw percolation process 
was introduced by Brummitt, Chatterjee, Dey and
Sivakoff~\cite{BrummittChatterjeeDeySivakoff15}
as a model for how a group of people might collectively solve a problem which
would be insurmountable individually. The premise is that each person has a piece
of a puzzle (or some knowledge, idea or expertise), and the pieces must be combined in
a certain way to solve the puzzle.

This is modelled using two graphs on a common vertex set: a red \emph{people graph},
with an edge if the two people know or collaborate with each other; and a blue
\emph{puzzle graph} if the two corresponding pieces of the puzzle can be combined.
If a pair of vertices are connected by both a red and a blue edge, the two
corresponding people share their information -- modelled in the graphs by
merging the two vertices into one cluster. 
Subsequently, two clusters are merged if a red and a blue edge runs between them.
Thus once parts of the puzzle have already been assembled, they become easier to merge.
The process continues until no additional
merges are possible. If it ends with one single cluster this indicates that the
puzzle has been solved, in which case we say that the process \emph{percolates}.
The process is formally defined in Section~\ref{sec:jigsaw}.

This process was first studied by Brummitt, Chatterjee, Dey and Sivakoff~\cite{BrummittChatterjeeDeySivakoff15},
and subsequently by Gravner and Sivakoff~\cite{GravnerSivakoff17}. They considered various deterministic
possibilities for the blue graph and random possibilities for the red graph and determined
some necessary and some sufficient conditions for the process to percolate
\emph{with high probability}, often abbreviated to \emph{whp},
meaning with probability tending to $1$ as the number of vertices $n$ tends to infinity.

Bollob\'as, Riordan, Slivken and Smith~\cite{BollobasRiordanSlivkenSmith17} then considered the case when both red and blue
graphs are random. Given a natural number $n$ and a real number
$p\in [0,1]$, the Erd\H{o}s-R\'enyi binomial random graph $G(n,p)$ is a graph on vertex set
$[n]:=\{1,2,\ldots,n\}$ in which each pair of vertices forms an edge with probability
$p$ independently.
Consider the binomial random graphs $G_1=G(n,p_1)$ and independently $G_2=G(n,p_2)$ on the same vertex set.
The random double graph created in this way
is denoted by $G(n,p_1,p_2)$. For brevity, we refer to the \emph{jigsaw process}
rather than the jigsaw percolation process.

\begin{thm}[Bollob\'as, Riordan, Slivken, Smith~\cite{BollobasRiordanSlivkenSmith17}]\label{thm:BRSS}
There exists a constant $c$ such that the following holds.
\begin{enumerate}
\item \label{thm:brss:subcrit} If $p_1p_2 < \frac{1}{cn\ln n}$, then with high probability the jigsaw process on
$G(n,p_1,p_2)$ does not percolate.
\item \label{thm:brss:supercrit} If $p_1p_2 > \frac{c}{n\ln n}$ and $p_1,p_2\ge \frac{c\ln n}{n}$, then with high probability the jigsaw process on $G(n,p_1,p_2)$ percolates.
\end{enumerate}
\end{thm}

In other words, percolation of the jigsaw process undergoes a phase transition
when the product $p_1p_2$ has order $\frac{1}{n\ln n}$. 
Note that connectedness
of each graph is a necessary condition for percolation, which is the reason
for the additional assumption in the supercritical case (Statement~\ref{thm:main:supercrit}): both $p_1$ and $p_2$
must be larger than $\frac{\ln n}{n}$, which is the threshold for connectedness, as first
proved by \Erdos\ and \Renyi~\cite{ErdosRenyi59}.

Indeed Buldyrev, Parshani, Paul, Stanley and Havlin \cite{BPPSH10} considered a related process, in which a set of vertices percolates if the graph spanned by these vertices is connected both in the red and the blue graph. 

Theorem~\ref{thm:BRSS} has subsequently been extended in various directions. Bollob\'as, Cooley, Kang
and Koch~\cite{BCKK17} proved a generalisation to $k$-uniform hypergraphs
and a jigsaw percolation process on the $j$-sets for each $1\le j \le k-1$.
In another direction, Cooley and Guti\'errez~\cite{CG17} proved an analogous result
for a set of $r$ graphs on a common vertex set, where $2\le r =o(\sqrt{\ln \ln n})$.

\subsection{Main theorem}
All of the results previously described (except for two special cases in~\cite{GravnerSivakoff17}) determine their thresholds only up to
a multiplicative constant.
In this paper we strengthen the result of Bollob\'as,
Riordan, Slivken and Smith~\cite{BollobasRiordanSlivkenSmith17} by determining the precise location of the threshold.

\begin{thm}\label{thm:main}
Let $\eps>0$ be any constant.
\begin{enumerate}
\item \label{thm:main:subcrit} If $p_1p_2\le \frac{1-\eps}{4n\ln n}$,
then with high probability the jigsaw process on $G(n,p_1,p_2)$ does not percolate.
\item \label{thm:main:supercrit} If $p_1p_2\ge \frac{1+\eps}{4n\ln n}$ and $p_1,p_2 \geq \frac{\ln n}{n}$,
then conditioned on $G_1,G_2$ being connected, with high probability the jigsaw process on $G(n,p_1,p_2)$ percolates.
\end{enumerate}
\end{thm}

Let us observe that the jigsaw process has a natural generalisation
to any number of graphs on a common vertex set (see~\cite{CG17}), and in particular the analogous
process on just one graph would percolate if and only if the graph is
connected. Thus we may view jigsaw percolation as a measure of whether
two graphs on a common vertex set are \emph{jointly connected}. In this way,
Theorem~\ref{thm:main} may be considered a double-graph analogue of the classical
result of Erd\H{o}s and R\'enyi~\cite{ErdosRenyi59} on the threshold for connectedness
of a random graph.

\subsection{Proof strategy}

For random graphs, the famous hitting time result of
Bollob\'as and Thomason~\cite{BollobasThomason85} relates the threshold for connectedness
of a random graph to the disappearance of the last isolated vertex,
implying that the critical obstruction for connectedness of random graphs
is the minimal one, and in particular, whether a random graph is connected
or not is essentially determined by local conditions.
In contrast,
the critical obstructions for jigsaw percolation on at least two
graphs are \emph{not} local ones -- in fact they are of size $\Theta(\ln n)$.
This makes determining the threshold,
and the proofs of both subcritical and supercritical cases, significantly
more complex.

The proof strategies for both the subcritical and supercritical cases of Theorem~\ref{thm:main} are
influenced by the fact that there is a \emph{bottleneck} in the jigsaw process.
More precisely, if any cluster reaches size around $\frac{1}{2np_1p_2}$, 
then it is large enough that whp it will go on to
incorporate all vertices. However, for small $p_1p_2$, no cluster will
reach this size. In fact, the size of the largest cluster
is approximately the smallest positive solution of the implicit equation
$2xN e^{-xN}=n^{-1/x}$, where $N=np_1p_2$,
which reaches $\frac{1}{2np_1p_2}$ when $p_1p_2$ is
$\frac{1}{4n\ln n}$, i.e.\ at the threshold for percolation.
Thus there is a bottleneck in the process at size around $2\ln n$.

Therefore in the subcritical case we will prove that whp no percolating set has size at least
$2\ln n$. On the other hand, in the supercritical case, the main difficulty is to
show that whp some percolating sets reach size slightly larger than $2\ln n$,
after which it is relatively straightforward to show that in fact whp one of these sets will also percolate
with all remaining vertices.

\subsection{The jigsaw process}\label{sec:jigsaw}

We now formally introduce the jigsaw process. A \emph{double graph} is a triple $(V,E_1,E_2)$,
where $V$ is a set of vertices and for $i=1,2$ we have $E_i\subset \binom{V}{2}$.
In other words, $(V,E_1)$ and $(V,E_2)$ are both graphs on a common vertex set.

\vspace{0.2cm}

\begin{algorithm}[H]\label{alg:jigsaw}
\SetAlgoLined
\textbf{Input:} Double graph $(V,E_1,E_2)$.\\
Set $i=0$, $V^{(0)}=V$, $E_1^{(0)}=E_1$ and $E_2^{(0)}=E_2$.\\
Set $H^{(0)}$ to be the auxiliary graph $(V^{(0)},E_1^{(0)}\cap E_2^{(0)})$.\\
\While{$E(H^{(i)})\neq \emptyset$}{
Set $V^{(i+1)}$ to be the set of components of $H^{(i)}$.\\
For $C,D\in V^{(i+1)}$ and $j=1,2$ let $\{C,D\}\in E_j^{(i+1)}$ iff there is at least one edge between $C$ and $D$ in $E_j^{(i)}$.\\
Set $H^{(i+1)}$ to be the auxiliary graph $(V^{(i+1)},E_1^{(i+1)}\cap E_2^{(i+1)})$.\\
Proceed to step $i+1$.
}

\textbf{Output:} $V^{(i)}$.
 \caption{{\sc Jigsaw Process}}
\end{algorithm}

\vspace{0.2cm}

Note that the process always terminates since $|V^{(i)}|$ is always positive, but strictly decreasing with $i$.
We say that the jigsaw process \emph{percolates} if at the end of the process $V^{(i)}$ contains exactly one vertex.

\subsection{Paper Overview}

The proofs of both the subcritical and the supercritical case are based on showing that,
for $k$ in a suitable range,
the probability that there exists a percolating set of size $k+1$ is approximately
\begin{equation*}\label{eq:heuristic}
\left(2kn\right)^{k+1} (p_1p_2)^k e^{-k^2 np_1p_2} e^{o(k)} = \left(2knp_1p_2 e^{-knp_1p_2}\right)^k n e^{o(k)} .
\end{equation*}
Very roughly, $(2kn)^{k+1}$ is the number of configurations on $k+1$ vertices (within a set of $n$ vertices) that can
make these $k$ vertices a percolating set, $(p_1p_2)^k$ is the probability that the relevant edges are present,
while $e^{-k^2np_1p_2}$
is the probability that this percolating set would actually be obtained (with an appropriate algorithm)
without other vertices also being added.

The main difficulty in each of the proofs is rigorously proving that this approximation is valid, as a lower bound for the
supercritical case and as an upper bound for the subcritical case.

Some preliminary results and notation are established in Section~\ref{sec:prel}.
This is followed by the proof of the subcritical regime in Section~\ref{sec:subcrit}
and by the proof of the supercritical regime in Section~\ref{sec:supercrit}.
Finally, in Section~\ref{sec:concluding} we discuss some further results and open problems.

\section{Preliminaries}\label{sec:prel}

We first collect various auxiliary results and definitions that we will need throughout the paper.
First note that since the jigsaw process is symmetric in the two graphs, we may assume without loss
of generality that $p_2\leq p_1\le 1$.
Furthermore, since percolation of the jigsaw process is a monotone property of double graphs,
we may also assume that 
\begin{equation}\label{eq:threshold}
p_1p_2=(1\pm \eps)\frac{1}{4n\ln{n}},
\end{equation} 
i.e.\ we assume for the subcritical case that $p_1p_2=(1- \eps)\frac{1}{4n\ln{n}}$, and for the
supercritical case that $p_1p_2=(1+ \eps)\frac{1}{4n\ln{n}}$.
Furthermore, since connectedness of both graphs is a necessary condition for the double graph
to percolate, in both subcritical and supercritical cases we may assume that
\begin{equation}\label{eq:conn}
p_1 \ge p_2\ge \frac{\ln n -\ln \ln n}{n}.
\end{equation}
This is already true by assumption in the supercritical case. In the subcritical case,
if $p_2$ does not satisfy this condition, then with high probability $G(n,p_2)$ is not
connected by the classical result of \Erdos\ and
\Renyi~\cite{ErdosRenyi59}, and therefore
the conclusion of Theorem~\ref{thm:main}~\ref{thm:main:subcrit} certainly holds.

Note that \eqref{eq:threshold} and \eqref{eq:conn} imply an upper bound on the individual probabilities namely
\begin{equation}\label{eq:upperprob}
p_1,p_2=O\left(\frac{1}{(\ln{n})^2}\right).
\end{equation}

Furthermore, observe that 
\begin{align}\label{pre:p2upper}
p_2 \leq \sqrt{p_1 p_2} = \sqrt{\frac{1+\eps}{4n\ln n}} \leq n^{-1/2}
\end{align}
and 
\begin{align}\label{pre:p1lower}
p_1 \geq \sqrt{p_1 p_2} \ge \frac{\sqrt{1-\eps}}{2\sqrt{n\ln n}}.
\end{align}

We will need to bound various random variables from above and below,
which we do by means of stochastic domination.

\begin{definition}
Let $X,Y$ be two positive integer-valued random variables. We say that $X$ \textit{stochastically dominates} $Y$,
and write $X \succ Y$, if $ \pr [X\geq r] \geq \pr [Y\geq r] $ for all $r\in \NN.$
\end{definition}

We will often use the following form of the Chernoff bound (see e.g.~\cite{JansonLuczakRucinskiBook}).

\begin{lem}\label{lem:chernoff}
For any binomial random variable $X$ we have
$$\pr [X\ge \mathbb{E}[X]+t]\le \exp\left(-\frac{t^2}{2(\mathbb{E}[X]+t/3)}\right)$$
and
$$\pr [X\le \mathbb{E}[X]-t]\le \exp\left(-\frac{t^2}{2\mathbb{E}[X]}\right).$$
\end{lem}

Throughout the paper we will ignore floors and ceilings when this does not
significantly affect the argument.
We will use the following bounds on factorials which hold for every positive integer (see e.g.~\cite{MR0069328}):
\begin{equation}\label{eq:stirling}
\left(\frac{n}{e}\right)^n \le \sqrt{2\pi n}\left(\frac{n}{e}\right)^n \leq n! \leq e \sqrt{ n}\left(\frac{n}{e}\right)^n.
\end{equation}

The following will be a central definition in the paper.

\begin{definition}
A \emph{percolating set} in a double graph $(V,E_1,E_2)$ is a set of vertices $U\subset V$ such that given the two edge sets
$E_i'=E_i \cap \binom{U}{2}$ for $i=1,2$, the jigsaw process on the double graph $(U,E_1',E_2')$ percolates.
\end{definition}

Whenever we talk about a \emph{cluster} of vertices, in particular this is always a percolating set.

\section{Subcritical Case: Proof of Theorem~\ref{thm:main}~\ref{thm:main:subcrit}}\label{sec:subcrit}

\subsection{Outline}
As mentioned previously, to prove Statement~\ref{thm:main:subcrit} of Theorem~\ref{thm:main}, we will show that
whp there is no percolating set of size at least $2\ln n$.
The key idea is to bound the number of configurations which can cause a set
of vertices to percolate.

\begin{definition}
A \emph{minimal percolating configuration}, is a percolating double-graph $(U,E_1,E_2)$,
where $U\subset [n]$ and $E_i\subset E(G_i)$ for $i=1,2$, and 
each $E_i$ forms a spanning tree in $U$.
\end{definition}
In other words, a minimal percolating configuration contains only the edges which are
needed for it to percolate. Note that we do not forbid additional edges in the host double graph $G(n,p_1,p_2)$,
but they are not part of the minimal percolating configuration.

It is easy to see by induction that any percolating set admits a minimal percolating configuration,
since for two clusters to merge it is enough that there is just one red edge and one blue edge between
them.

In order to bound the number of minimal percolating configurations, we will
analyse how the jigsaw process might evolve on them by introducing the
\emph{absorption process} in Section~\ref{sec:absorption}.
In Section~\ref{sec:configbounds}, we will
characterise minimal percolating configurations according to certain
parameters related to their corresponding absorption processes,
and state bounds on the number of possibilities for configurations
based on these parameters (Theorem~\ref{thm:configswithr} and Lemma~\ref{lem:configsnor}).
These bounds will be proved in Section~\ref{sec:configsnor} and~\ref{sec:configswithr}, after some
technical preliminaries have been proved in Section~\ref{sec:subcritprelim}.
Finally, in Section~\ref{sec:pfsubcrit}, we show how these bounds
prove Theorem~\ref{thm:main}~\ref{thm:main:subcrit}.

\subsection{The absorption process}\label{sec:absorption}

We need a variant of the jigsaw process, which we call
an \emph{absorption process}. In this process we gradually construct a percolating set $S_i=\{v_1,\ldots, v_{t(i)}\}$. 

\vspace{0.2cm}
\begin{algorithm}[H]
\SetAlgoLined
\textbf{Input:} Double graph $(V,E_1,E_2)$, vertex $v_1\in V$ and a set of clusters $\cC$ partitioning $V\setminus \{v_1\}$.\\
Set $i=1$, $t(i)=1$, $\cC_1=\cC$ and $S_1=\{v_1\}$.\\
\While{$t(i)\ge i$}{
Set $\cC'_i\subset \cC_i$ be the set of clusters of size at most $t(i)$ which are adjacent to $v_i$ in one colour and to some vertex from $\{v_1,\ldots,v_i\}$ in the other colour. \\
Set $S_{i+1}=S_i\cup \left(\bigcup_{C\in \cC_i'}C\right)$.\\
Set $t(i+1)=|S_{i+1}|$.\\
Set $v_{t(i)+1},\ldots,v_{t(i+1)}$ to be the vertices of $S_{i+1}\setminus S_i$ in any order.\\
Set $\cC_{i+1}=\cC_{i}\setminus \cC_{i}'$.\\
Proceed to step $i+1$.
}

\textbf{Output:} $S_i$.
 \caption{{\sc Absorption Process}}
\end{algorithm}
\vspace{0.2cm}

If at the end of this algorithm we have $S_i=V$, we say that the absorption
process \emph{percolates}.
If the double graph $(V,E_1,E_2)$ is clear from the context, then we sometimes abuse terminology
slightly by referring to $(v_1,\cC)$ as the input of the algorithm.

\begin{lem}\label{lem:canconstruct}
For every percolating double-graph $(V,E_1,E_2)$ there exists a vertex $v_1$ and a set
of disjoint clusters $\cC$ such that the absorption process with input $(v_1,\cC)$ percolates.
\end{lem}

\begin{proof}
We prove this statement by induction on the size of $V$. Clearly if $|V|=1$ the statement holds, so
assume that it holds for every set of size at most $k$ and that $|V|=k+1$. 

Since $(V,E_1,E_2)$ percolates, in the final step of the jigsaw process,
a connected auxiliary graph $H^{(i)}$ was merged into one cluster. Consider a vertex of 
$H^{(i)}$ which is not a cutvertex. This vertex corresponds to a
set $X$ of vertices in $V$, and let $Y:=V\setminus X$. Then we have partitioned
$V$ into two non-empty percolating sets.

Fix edges $e_1=x_1y_1\in E_1$ and $e_2=x_2y_2\in E_2$ such that $x_1,x_2\in X$
and $y_1,y_2\in Y$. Note that since $(V,E_1,E_2)$ percolates, such edges must exist.

By the induction
hypothesis, there exists a vertex $v_X\in X$ and a set of clusters $\cC_X$ within $X$
such that the absorption process on $X$ with input $(v_X,\cC_X)$ percolates. Let us denote by $X_{i}$
the percolating set constructed within $X$ after $i$ steps of this absorption process.
Let $i'$ be the first step in which $X_{i'}$ contains both $x_1$ and $x_2$.
We define $v_Y, \cC_Y, j'$ and $Y_{j'}$ analogously within $Y$.

Without loss of generality we have $|X_{i'}|\ge |Y_{j'}|$. Recall that the vertices of $X_{i'}$ are
ordered in the absorption process, and let $i''$ be the index of the later of $x_1,x_2$, wlog $x_2$.
Then when we reach $x_2$ at step $i''$, we have $|X_{i''}|\ge |X_{i'}|\ge |Y_{j'}|$, and therefore at
step $i''$, the cluster $Y_{j'}$ can be merged with $X_{i''}$. Thus the absorption process on $V$
with input vertex $v_X$ and with input of clusters
$\cC_X, Y_{j'}$ and $\{C\in \cC_Y \mid C\cap Y_{j'}= \emptyset\}$ percolates.
\end{proof}

Lemma~\ref{lem:canconstruct} tells us that
any percolating set can be discovered via an absorption process with
some input of starting vertex and clusters.
Note that this is slightly non-constructive, since 
some percolating clusters are already in the input of the algorithm.

\subsection{Bounding configurations}\label{sec:configbounds}
Our aim is to bound the number of possible minimal percolating configurations
by analysing how the absorption process evolves on them.
For this analysis, we will need to define the order in which clusters are
added, which is primarily determined by the step in which a cluster is added,
but there may be more than one cluster added in a single step. In such a case
we order the clusters added in a single step according to the order of their
smallest vertices (recall that the vertices of $G(n,p_1,p_2)$ are labelled $1,\ldots,n$).

\begin{definition}
Given integers $k,\ell,r$,
let $\pc_{k,\ell,r}$ be the number of possible minimal percolating configurations on
vertex set $[k+r]$ for which there exists some input of starting vertex and clusters
such that the absorption process
with this input percolates in $\ell$ steps,
and which adds a cluster of size exactly $r$ in the $\ell$-th step (and therefore has
a percolating set of size
at most $k$ after $\ell-1$ steps).
\footnote{Note that there may also be other clusters
added in the $\ell$-th step.}
\end{definition}

The main difficulty in the subcritical case is to prove the following.

\begin{thm} \label{thm:configswithr}
Given integers $1\le r,\ell\le k$,
we have
$$\pc_{k,\ell,r}\le \binom{k+r}{r} (k!)^2 (r!)^2 2^{k+r} e^{r+\ell} \cdot \frac{k\ell r^3 e^{582}}{2}.$$
\end{thm}

In order to prove this, we first approximate a slightly different parameter.

\begin{definition}
Let $\pc'_{k,\ell}$ be the number of possible minimal percolating configurations on
vertex set $[k]$ for which there exists some input of starting vertex and clusters
such that the absorption process
with this input percolates in at most $\ell$ steps.
\end{definition}

Note that there are two crucial differences between this and the definition of $\pc_{k,\ell,r}$:
first, we do not consider the final cluster on $r$ vertices; and second, we do not demand
that the process takes \emph{exactly} $\ell$ steps to percolate.

Let us observe that
\begin{equation}\label{eq:configspartition}
\pc_{k,\ell,r} \le \binom{k+r}{r} \pc'_{k,\ell} 2r^2\ell \pc_{r,r}',
\end{equation}
since a configuration which contributes to $M_{k,\ell,r}$
can be partitioned into a minimal percolating set on $k$ vertices
whose absorption process takes at most $\ell$ steps to percolate
and a minimal percolating set on $r$ vertices, for which $r-1\le r$
is certainly an upper bound on the number of steps required for an absorption process on $r$ vertices
to percolate. There are $\binom{k+r}{r}$ possible ways of partitioning
the vertices and $\pc'_{k,\ell} \pc_{r,r}'$ possible minimal percolating configurations
on the two resulting parts. There must also be an edge of each colour
between the two vertex sets, one of which has the last vertex of
the set of size $k$ as a neighbour, which leaves $2r^2\ell$ choices
as claimed.
Thus in order to prove Theorem~\ref{thm:configswithr}, it is
enough to bound $\pc'_{k,\ell}$.

\begin{lem} \label{lem:configsnor}
Given integers $1 \le \ell \le k$, we have
$$\pc'_{k,\ell}\le (k!)^2 2^{k} e^{\ell} \cdot \frac{ke^{291}}{2} .$$
\end{lem}

We will prove Lemma~\ref{lem:configsnor} in Section~\ref{sec:configsnor}.
Subsequently, in Section~\ref{sec:configswithr} we will use Lemma~\ref{lem:configsnor}
to prove Theorem~\ref{thm:configswithr}. Finally we will show how Theorem~\ref{thm:main}~\ref{thm:main:subcrit}
follows from Theorem~\ref{thm:configswithr} in Section~\ref{sec:pfsubcrit}.
But first in Section~\ref{sec:subcritprelim} we will collect a few auxiliary results
that we will need for the proof of Lemma~\ref{lem:configsnor}.

\subsection{Auxiliary results}\label{sec:subcritprelim}

Our aim in Lemma~\ref{lem:configsnor} is to bound $\pc'_{k,\ell}$,
and we will need the following basic bound on
$\pc'_{j,j}$,
i.e.\ the number of minimal percolating
configurations on $j$ vertices, with no restrictions
on the number of steps they take to percolate in an absorption process
(since certainly $j-1\le j$ is an upper bound on the
number of steps an absorption process on $j$ vertices can take to percolate).
The following result was already proved in \cite{BollobasRiordanSlivkenSmith17}, but
since the proof is easy, we include it here for completeness.

\begin{claim}[\cite{BollobasRiordanSlivkenSmith17}]\label{claim:spantrees}
For any integer $j\ge 1$, we have
$\pc_{j,j}' \le j^{2j-4}$.
\end{claim}

\begin{proof}
Recall that in a minimal percolating configuration, the red and the blue edge sets each form a spanning tree.
By Cayley's formula there are $j^{j-2}$ spanning trees on $j$ vertices and the result follows.
\end{proof}

In the proof of Lemma~\ref{lem:configsnor}, we also use the following technical proposition.

\begin{prop}\label{prop:largej}
For every $j\geq 3$, we have
$$
\sum_{i=j}^\infty \frac{1}{(i+1)(i+2)\ldots(i+j-1)}=\frac{1}{(j-2)}\cdot \frac{j!}{(2j-2)!}.
$$
In particular,
$$
\sum_{i=j}^\infty \frac{1}{(i+1)(i+2)\ldots(i+j-1)}\leq e^2\frac{j!}{j-2} \left(\frac{e}{2j}\right)^{2j-2}.
$$
\end{prop}

\begin{proof}
The Chu-Vandermonde identity, which can be easily verified combinatorially,
states that for non-negative integers $a,b,c$ we have
$$\binom{a+b}{c}=\sum_{\ell=0}^{c}\binom{a}{\ell}\binom{b}{c-\ell}.$$
In fact this equality also holds if $a$ is negative, where we interpret
$\binom{a}{\ell}$ as $\frac{a(a-1)\ldots(a-\ell+1)}{\ell!}$ (see eg. Corollary 2.2.3 in  \cite{MR1688958}).
Setting $a=-i-1$, $b=i+j-1$, and $c=j-2$ leads to
\begin{align*}
1&=\sum_{\ell=0}^{j-2}(-1)^{\ell}\frac{(i+1)(i+2)\ldots (i+\ell)}{\ell!}\cdot\frac{(i+j-1)!}{(j-2-\ell)! (i+\ell+1)!}\\
&=\sum_{\ell=0}^{j-2}(-1)^{\ell}\frac{(i+\ell)!}{\ell!}\cdot \frac{(i+1)(i+2)\ldots (i+j-1)}{(j-2-\ell)! (i+\ell+1)!}\\
&=\frac{1}{(j-2)!}\sum_{\ell=0}^{j-2}(-1)^{\ell}\frac{(j-2)!}{\ell! (j-2-\ell)!}\cdot \frac{(i+1)(i+2)\ldots (i+j-1)}{i+\ell+1}
\end{align*}
implying
\begin{equation}\label{eq:pfd}
\frac{1}{(i+1)(i+2)\ldots(i+j-1)}=\frac{1}{(j-2)!}\sum_{\ell=0}^{j-2}(-1)^{\ell}\binom{j-2}{\ell}\frac{1}{i+1+\ell}.
\end{equation}

Summing this for $j \le i \le m$ leads to
\begin{align*}
\sum_{i=j}^{m}\frac{1}{(i+1)(i+2)\ldots(i+j-1)}\hspace{-1.5cm} & \\
&=\frac{1}{(j-2)!}\sum_{i=j}^{m}\sum_{\ell=0}^{j-2}(-1)^{\ell}\binom{j-2}{\ell}\frac{1}{i+1+\ell}\\
&=\frac{1}{(j-2)!}\sum_{i=j}^{m}\left(\frac{1}{i+1} + \frac{(-1)^{j-2}}{i+j-1}
+\sum_{\ell=1}^{j-3}(-1)^{\ell}\left(\binom{j-3}{\ell}+\binom{j-3}{\ell-1}\right)\frac{1}{i+1+\ell}\right)\\
&=\frac{1}{(j-2)!}\sum_{i=j}^{m}\sum_{\ell=0}^{j-3}(-1)^{\ell}\binom{j-3}{\ell}\left(\frac{1}{i+1+\ell}-\frac{1}{i+2+\ell}\right)\\
&=\frac{1}{(j-2)!}\left(\sum_{\ell=0}^{j-3}(-1)^{\ell}\binom{j-3}{\ell}\frac{1}{j+1+\ell}-\sum_{\ell=0}^{j-3}(-1)^{\ell}\binom{j-3}{\ell}\frac{1}{m+2+\ell}\right)
\end{align*}
implying
$$\sum_{i=j}^{\infty}\frac{1}{(i+1)(i+2)\ldots(i+j-1)}=\frac{1}{(j-2)!}\sum_{\ell=0}^{j-3}(-1)^{\ell}\binom{j-3}{\ell}\frac{1}{j+1+\ell}.$$

We now apply~\eqref{eq:pfd} again, with $j$ replaced by $j-1$ and $i$ replaced by $j$, to obtain
$$\frac{1}{(j-2)!}\sum_{\ell=0}^{j-3}(-1)^{\ell}\binom{j-3}{\ell}\frac{1}{j+1+\ell}=\frac{1}{j-2}\cdot\frac{1}{(j+1)(j+2)\ldots (2j-2)}=\frac{1}{j-2}\cdot\frac{j!}{(2j-2)!}
$$
as required.

For the second statement, we simply apply \eqref{eq:stirling} to obtain
$$\frac{1}{(2j-2)!}\leq \left(\frac{e}{2j-2}\right)^{2j-2}= \left(\frac{j}{j-1}\right)^{2(j-1)} \left(\frac{e}{2j}\right)^{2j-2} \leq e^2 \left(\frac{e}{2j}\right)^{2j-2}$$
and the result follows.
\end{proof}

\subsection{Proof of Lemma~\ref{lem:configsnor}}\label{sec:configsnor}

We aim to prove Lemma~\ref{lem:configsnor}, i.e.\ that $\pc'_{k,\ell}\le k (k!)^2 2^{k-1} e^{\ell+291} $
for $\ell\le k$.
Recall that when a cluster $C$ is added in step $i$ of the absorption process, the vertex $v_i$ will be joined to
a vertex $w_C$ of $C$ in one colour, and for some $i'\le i$, the vertex $v_{i'}$
will be joined to a vertex $w_C'$ of $C$ in the other colour.

For $1\le j \le k/2$, let $c_j$ be the number of clusters of size $j$
which are added in the process and set $\mathbf{c}=(c_1,\ldots,c_{k/2})$.
Observe that no clusters of size larger than $k/2$ can be added, since clusters can only
be added to a percolating set which is at least as large as the cluster, and we have only $k$ vertices in total.
Note also that the first vertex $v_1$
does not count towards $c_1$ (recall that it is not considered a cluster), and so
\begin{equation}\label{eq:totalclustsize}
\sum_{j=1}^{k/2} jc_j=k-1.
\end{equation}

Initially, we order the clusters $C_1,\ldots,C_d$, where $d:=\sum_{j=1}^{k/2} c_j$ is the
total number of clusters, according to the order of their smallest vertices.
Recall that the absorption process gives us a new order on the clusters
according to the order in which they are added,
so we have a permutation $\sigma$ on $[d]$ such that $C_{\sigma(s)}$ is the
$s$-th cluster to be added.

We further define the following parameters:
\begin{itemize}
\item Let $i_s$ denote the step in which $C_s$
is added.
\item Let $j_s$ denote the size of $C_s$.
\item Let $m_s:=\max(i_s,j_s)$.
\end{itemize}
In particular, this gives a vector $\mathbf{i}=(i_1,\ldots,i_{d})$.
Let $I:=I(\mathbf{c})$ denote the set of
\emph{permissible} vectors, meaning that
\begin{enumerate}[(P-i)]
\item \label{permcond:clustersoon} 
$i_{\sigma(s)}\le 1+\sum_{t=1}^{s-1}j_{\sigma(t)}$ for all $1\le s \le d$, i.e.\ we do not run out of vertices before adding the next cluster.
\item \label{permcond:clustersmall} 
$j_{\sigma(s)}\le 1+\sum_{t=1}^{s-1}j_{\sigma(t)}$, for all $1\le s \le d$, i.e.\ we do not add a cluster larger than the current percolating set.
\item \label{permcond:percolatingsteps} 
$i_s\le \ell$ for all $1\le s \le d$.
\end{enumerate}

Note that Conditions~(P-\ref{permcond:clustersoon}) and~(P-\ref{permcond:clustersmall}) are
implicitly dependent on the vector $\mathbf{i}$ because $\sigma$ is dependent on $\mathbf{i}$.
Note also that while (P-\ref{permcond:clustersmall}) is \emph{necessary} to ensure that we do not add
a cluster larger than the current percolating set, it is not quite \emph{sufficient} as if we add
multiple clusters to a percolating set in a single step, this condition would allow for all but the first
of these clusters to be too large. However, since we are concerned with upper bounds, this is
not a problem.

We may now bound $\pc_{k,\ell}'$ by considering how many choices we have for various parameters
for each fixed $\mathbf{c}=(c_1,\ldots,c_{k/2})$ and $\mathbf{i}=(i_1,\ldots,i_d)$ and summing over all possibilities for $\mathbf{c},\mathbf{i}$.
Given vectors $\mathbf{c}$ and $\mathbf{i}$,:

\begin{itemize}
\item We have $k$ choices for the first vertex $v_1$;
\item There are $\frac{(k-1)!}{\prod_{j=1}^{k/2} (j!)^{c_j} c_j!}$ distinct ways of assigning the remaining vertices to clusters;
\item For each cluster $C_s$ of size $j_s$ which is to be added in step $i_s$ we have at most:
\begin{itemize}
\item $j_s^2$ choices for the two vertices $w_C,w'_C$;
\item $2$ choices for the colour of the edge from $v_{i_s}$ to $w_C$;
\item $\pc_{j_s,j_s}' \le j_s^{2j_s-4}$ possible minimal percolating configurations within $C$
(by Claim~\ref{claim:spantrees});
\item $i_s$ choices for $i_s'\le i_s$.
\end{itemize}
\end{itemize}
Thus we obtain
\begin{equation}\label{eq:contributions}
\pc'_{k,\ell} \le \sum_{\mathbf{c}} \left\{k\cdot \frac{(k-1)!}{\prod_{j=1}^{k/2} (j!)^{c_j} c_j!} \cdot
\left(\prod_{j=1}^{k/2} (2 j^2  j^{2j-4})^{c_j}\right)
\cdot
\sum_{\mathbf{i}\in I} \prod_{s=1}^{d} i_s \right\}.
\end{equation}
We first consider the terms involving $i_s$, which require the most care. We have
\begin{equation}\label{eq:prodis}
\sum_{\mathbf{i}\in I} \prod_{s=1}^{d} i_s
\le \sum_{\mathbf{i}\in I} \prod_{s=1}^{d} m_s
= \sum_{\mathbf{i}\in I} \frac{\prod_{s=1}^{d} m_s(m_s+1)\ldots (m_s+j_s-1)}{\prod_{s=1}^{d}(m_s+1)\ldots (m_s+j_s-1)}.
\end{equation}
Here we use the convention that an empty product is interpreted as $1$.
We bound the numerator in~\eqref{eq:prodis} with the following claim.

\begin{claim}\label{claim:factorial}
$\prod_{s=1}^{d} m_s(m_s+1)\ldots (m_s+j_s-1)\le (k-1)!$
\end{claim}
\begin{proof}
The conditions~(P-\ref{permcond:clustersoon}) and~(P-\ref{permcond:clustersmall}) together imply that
$m_{\sigma(s)}\le 1+\sum_{t=1}^{s-1}j_{\sigma(t)}$, and therefore
\begin{align*}
\prod_{s=1}^d \left\{m_s(m_s+1)\ldots (m_s+j_s-1)\right\}
& = \prod_{s=1}^d \left\{m_{\sigma(s)}(m_{\sigma(s)}+1)\ldots (m_{\sigma(s)}+j_{\sigma(s)}-1)\right\}\\
& \le \prod_{s=1}^d \left\{\left(\sum_{t=1}^{s-1}j_{\sigma(t)}+1\right)\left(\sum_{t=1}^{s-1}j_{\sigma(t)}+2\right)\ldots \left(\sum_{t=1}^{s-1}j_{\sigma(t)}+j_{\sigma(s)}\right)\right\}\\
& = \left(\sum_{s=1}^d j_{\sigma(s)}\right)! = (k-1)!
\end{align*}
as claimed.
\end{proof}

We handle the denominator and sum in~\eqref{eq:prodis} with the following claim.

\begin{claim}\label{claim:trickysum}
	$$\sum_{\mathbf{i}\in I} \frac{1}{\prod_{s=1}^{d} (m_s+1)(m_s+2)\ldots (m_s+j_s-1)}
	\le \ell^{c_1}(2 \ln (\ell+1))^{c_2} 
	\prod_{j=3}^{k/2} \left(3 e^2j!\left(\frac{e}{2j}\right)^{2j-2}\right)^{c_j}.
	$$
\end{claim}
\begin{proof}
Let us note that we \emph{first} fixed the assignment of vertices to clusters, which in particular
determines the $j_s$, and \emph{then} chose the vector $\mathbf{i}$, so in particular $j_s$
is not dependent on $i_s$ (although which values of $i_s$ are permissible does depend on the $j_s$).

	Therefore we have
	\begin{align*}
	\sum_{\mathbf{i}\in I} \frac{1}{\prod_{s=1}^{d} (m_s+1)(m_s+2)\ldots (m_s+j_s-1)} \hspace{-2cm} & \\
	& \le \sum_{i_1=1}^\ell\ldots \sum_{i_d=1}^\ell \prod_{s=1}^{d} \frac{1}{ (m_s+1)(m_s+2)\ldots (m_s+j_s-1)}\\
	& \le \sum_{i_1=\min(j_1,\ell)}^\ell\ldots \sum_{i_d=\min(j_d,\ell)}^\ell \prod_{s=1}^{d} \frac{j_s}{ (m_s+1)(m_s+2)\ldots (m_s+j_s-1)}\\
	&\le \prod_{j=1}^{\ell} \left(\sum_{i=j}^\ell\frac{j}{(i+1)\ldots(i+j-1)}\right)^{c_j}\prod_{j=\ell+1}^{ k/2} \left(\frac{j}{(j+1)\ldots(2j-1)}\right)^{c_j}\\
	&\le \prod_{j=1}^{ k/2 } \left(\sum_{i=j}^{\ell+j-1}\frac{j}{(i+1)\ldots(i+j-1)}\right)^{c_j}.	
	\end{align*}
	In the third inequality we have used $m_s \ge i_s,j_s$.
	In the last inequality we have used the fact that $\ell+j-1\ge \ell,j$.
	
	The $j=1$ term in the product is simply
	\begin{equation}\label{eq:j=1}
	\left(\sum_{i=1}^\ell 1\right)^{c_1}= \ell^{c_1}.
	\end{equation}
	We bound the term $j=2$ as follows:
	\begin{equation}\label{eq:j=2}
	\left(\sum_{i=2}^{\ell+1}\frac{2}{i+1}\right)^{c_2} \le \left(\int_{1}^{\ell+1} \frac{2}{x} dx\right)^{c_2}
	= \left( 2 \ln (\ell+1) \right)^{c_2}.
	\end{equation}
	For $j\ge 3$, we apply Proposition~\ref{prop:largej} to obtain
	\begin{equation*}
	\sum_{i=j}^{\infty}\frac{j}{(i+1)\ldots (i+j-1)}
	\le je^2 \frac{j!}{j-2}\left(\frac{e}{2j}\right)^{2j-2}
	\leq 3 e^2 j!\left(\frac{e}{2j}\right)^{2j-2},
	\end{equation*}
	which gives
	\begin{equation}\label{eq:j=3tol}
	\prod_{j=3}^{ k/2} \left(\sum_{i=j}^{\ell+j-1}\frac{j}{(i+1)\ldots (i+j-1)}\right)^{c_j}
	\leq \prod_{j=3}^{ k/2 }\left( 3 e^2 j!\left(\frac{e}{2j}\right)^{2j-2}\right)^{c_j}.
	\end{equation}

	Combining the bounds from~\eqref{eq:j=1}, \eqref{eq:j=2} and~\eqref{eq:j=3tol} proves the claim.
\end{proof}

Substituting the bounds from Claims~\ref{claim:factorial} and~\ref{claim:trickysum} into~\eqref{eq:prodis},
we have
$$
\sum_{\mathbf{i}\in I} \prod_{s=1}^{d} i_s \le (k-1)!\ell^{c_1}(2 \ln (\ell+1))^{c_2} 
\prod_{j=3}^{k/2} \left(3 e^2j!\left(\frac{e}{2j}\right)^{2j-2}\right)^{c_j},
$$
and therefore~\eqref{eq:contributions} gives
\begin{align*}
\pc'_{k,\ell} & \le \sum_{\mathbf{c}} \left\{ \frac{k!}{\prod_{j=1}^{k/2} (j!)^{c_j} c_j!} \cdot
\left(\prod_{j=1}^{k/2} (2 j^2  j^{2j-4})^{c_j}\right)
\cdot
(k-1)!\left( \ell^{c_1}(2\ln (\ell+1))^{c_2}
\prod_{j= 3}^{k/2} \left(3 e^2 j!\left(\frac{e}{2j}\right)^{2j-2}\right)^{c_j}\right)\right\}\\
& = k!(k-1)! \sum_{\mathbf{c}}
\left\{
\frac{(2\ell)^{c_1}}{c_1!} \cdot \frac{(16\ln (\ell+1))^{c_2}}{2^{c_2}c_2!}
\cdot \prod_{j= 3}^{k/2} \frac{\left(2j^{2j-2}\cdot 3e^2 j!\left(\frac{e}{2j}\right)^{2j-2}\right)^{c_j}
}{(j!)^{c_j}c_j!}
\right\}\\
& \stackrel{\eqref{eq:totalclustsize}}{=} k!(k-1)!2^{k-1} \sum_{\mathbf{c}}
\left\{
\frac{\ell^{c_1}}{c_1!} \cdot \frac{(2 \ln (\ell+1))^{c_2}}{c_2!}
\cdot \prod_{j= 3}^{k/2} \frac{1}{c_j!}\left(\frac{3e^2 \left(\frac{e}{2}\right)^{2j-2}
}{2^{j-1}}\right)^{c_j}
\right\}\\
& \le k!(k-1)!2^{k-1}
\left\{
\sum_{c_1=0}^\infty
\frac{\ell^{c_1}}{c_1!}
\right\}
\cdot
\left\{
\sum_{c_2=0}^\infty
\frac{(2 \ln (\ell+1))^{c_2}}{c_2!}
\right\}
\cdot
\left\{
\prod_{j= 3}^{k/2}
\sum_{c_j=0}^\infty
\frac{1}{c_j!}\left(3e^2 \left(\frac{e^2}{8}\right)^{j-1}\right)^{c_j}
\right\}\\
& = k!(k-1)!2^{k-1}
\exp\left(\ell\right)
\cdot
\exp\left(
2 \ln (\ell+1)
\right)
\cdot
\prod_{j= 3}^{k/2}
\exp\left(
3e^2 \left(\frac{e^2}{8}\right)^{j-1}\right)\\
& \le k (k!)^2 2^{k-1}
e^{\ell}
\exp\left(
\sum_{j= 3}^{k/2}
3e^2 \left(\frac{e^2}{8}\right)^{j-1}
\right).
\end{align*}
Note that in the last line we have used the fact that $\ell \le k-1$, since this is an upper
bound on the number of steps it can take for the absorption process on $k$ vertices to percolate.
We bound the remaining sum by
$$\sum_{j= 3}^{k/2}
3e^2 \left(\frac{e^2}{8}\right)^{j-1}\leq \frac{3e^2}{1-e^2/8}\leq 291,$$
since $e^2/8<1$.
Thus we obtain
\begin{align*}
\pc'_{k,\ell} \le (k!)^2 2^{k}e^{\ell}\cdot \frac{ke^{291}}{2},
\end{align*}
which completes the proof of Lemma~\ref{lem:configsnor}.\qed

\subsection{Proof of Theorem~\ref{thm:configswithr}}\label{sec:configswithr}

We can now prove Theorem~\ref{thm:configswithr}.
Observe that Lemma~\ref{lem:configsnor} gives a bound on $\pc_{j,j}'$
which is better than Claim~\ref{claim:spantrees} for large $j$:
\begin{equation}\label{eq:Xjbetterbound}
\pc'_{j,j} \le (j!)^2 2^{j}e^{j}\cdot \frac{je^{291}}{2}.
\end{equation}
We use~\eqref{eq:Xjbetterbound} and Lemma~\ref{lem:configsnor} to obtain
\begin{align*}
\pc_{k,\ell,r} & \stackrel{\eqref{eq:configspartition}}{\le} \binom{k+r}{r} \pc'_{k,\ell} 2r^2\ell \pc_{r,r}'\\
& \le \binom{k+r}{r} \left((k!)^2 2^{k}e^{\ell}\cdot\frac{ke^{291}}{2}\right) 2r^2\ell \left((r!)^2 2^{r}e^{r}\cdot\frac{re^{291}}{2}\right)\\
& = \binom{k+r}{r} (k!)^2 (r!)^2 2^{k+r} e^{r+\ell} \cdot \frac{k\ell r^3 e^{582}}{2},
\end{align*}
as claimed in Theorem~\ref{thm:configswithr}.\qed

\subsection{Proof of Theorem~\ref{thm:main}~\ref{thm:main:subcrit}}\label{sec:pfsubcrit}

To prove the subcritical case of Theorem~\ref{thm:main}, we need the following
strengthening of the notion of a minimal percolating configuration.

\begin{definition}
An \emph{optimal configuration} in a double graph $(V,E_1,E_2)$
is a minimal percolating configuration $(U,E_1',E_2')$, where $U\subset V$ and $E_i' \subset E_i$ for $i=1,2$,
together with a vertex $v\in U$ and a set of clusters partitioning $U\setminus \{v\}$ such that an 
absorption process with this input will percolate, and
furthermore the following holds.
Let $\ell$ be the number of steps it takes for this absorption process to percolate,
and let $v=v_1,\ldots,v_{|U|}$ be the vertices of $U$ in the order that they are added
to the percolating set in the absorption process.
Then no vertex in $V\setminus U$ has an edge to $\{v_1,\ldots,v_{\ell-1}\}$
in both $E_1$ and $E_2$.
\end{definition}
In other words, an optimal configuration is a minimal percolating configuration which allows
the absorption process to percolate and includes all of the vertices which
could be added as clusters in the process before step $\ell$. Note that it is \emph{not} necessarily
a maximal percolating set, since it is still possible that
some clusters of size larger than one could have been added to the
process, that more single vertices could have been added in step $\ell$,
or indeed that the absorption process could have continued beyond $\ell$ steps.

Set
$$k_0:= 2\ln n.$$
If $G(n,p_1,p_2)$ percolates,
then by Lemma~\ref{lem:canconstruct} there exists an input of starting vertex and clusters
such that running an absorption process with this input will lead to an optimal configuration
(and in particular a minimal percolating configuration)
of size greater than $k_0$.
Let us consider the first time at which this process becomes
larger than $k_0$, in step $\ell$, say. Then it reached size $k\le k_0$ in either the $(\ell-1)$-th or
the $\ell$-th step, and we next
added a cluster of size $r$ in the $\ell$-th step such
that $k+r> k_0$.
We aim to bound the number of optimal configurations with parameters $k,\ell,r$
and sum over all $k,\ell,r$, observing that
\begin{equation}\label{eq:klr}
1\le r, \ell \le k \le k_0 < k+r.
\end{equation}
For $0\le i \le \ell$, define $X_{i}$ to be the set of vertices in the percolating set after $i$
steps of the absorption process,
which is terminated the moment we reach size larger than $k_0$, so in particular
we have $|X_\ell|=k+r$, and furthermore $\ell \le |X_{\ell-1}|\le k$.

Since we have an optimal configuration, none of the vertices outside $X_\ell$
may have both a red and a blue neighbour within the first $\ell-1$ vertices $\{x_1,x_2,\ldots,x_{\ell-1}\}$ of $X_{\ell-1}$.
Note that this holds for a given vertex with probability at most
$$(1-p_1)^{\ell-1}+(1-p_2)^{\ell-1}-(1-p_1)^{\ell-1}(1-p_2)^{\ell-1}=1-(1-(1-p_1)^{\ell-1})(1-(1-p_2)^{\ell-1}).$$
By \eqref{eq:upperprob} we have $p_1(\ell-1)\le p_1 k_0 =o(1)$ and similarly $p_2 (\ell-1)=o(1)$.
Therefore
$$(1-(1-p_1)^{\ell-1})(1-(1-p_2)^{\ell-1})=(1+o(1))(\ell-1)^2p_1p_2.$$
Thus the probability that none of the vertices outside $X_\ell$ has both a red and a blue neighbour
within $\{x_1,\ldots,x_{\ell-1}\}$ is at most
\begin{align*}
\left(1-(1+o(1))(\ell-1)^2p_1 p_2\right)^{n-k-r}
&\leq \exp\left(-(1+o(1))(\ell-1)^2p_1p_2(n-k-r)\right)\\
&\leq \exp\left(-(1-\eps_0)(\ell-1)^2 n p_1 p_2\right),
\end{align*}
where $\eps_0:=\eps/3$.

Therefore the expected number of optimal configurations with parameters $k,\ell,r$ is at most
$$
\binom{n}{k+r}M_{k,\ell,r}(p_1p_2)^{k+r-1} \exp(-(1-\eps_0)(\ell-1)^2 np_1p_2).
$$
Let us define $S$ to be the set of triples $(k,\ell,r)$ satisfying~\eqref{eq:klr}. We need to bound the expression
\begin{align}\label{eq:optconfigs}
& \sum_{(k,\ell,r)\in S}\binom{n}{k+r}M_{k,\ell,r}(p_1p_2)^{k+r-1} \exp(-(1-\eps_0)(\ell-1)^2 np_1p_2)\nonumber\\
\le  & \sum_{(k,\ell,r)\in S} \frac{n^{k+r}}{(k+r)!}\binom{k+r}{r} (k!)^2 (r!)^2 2^{k+r} e^{r+\ell}
\cdot \frac{k\ell r^3 e^{582}}{2} \cdot
\left(\frac{1-\eps}{4n\ln n}\right)^{k+r-1} \exp(-(1-\eps_0)(\ell-1)^2 np_1p_2)\nonumber\\
\le  & \frac{ e^{582}\cdot n\cdot 2\ln n}{1-\eps} \left(\sum_{k=k_0/2}^{k_0}k \cdot k! \left(\frac{1-\eps}{2\ln n}\right)^{k}
\sum_{r=k_0-k}^{k} r! r^3\left(\frac{e(1-\eps)}{2\ln n}\right)^r \right)
\sum_{\ell=1}^{k_0}\ell \exp\left(\ell -(1-\eps_0)(\ell-1)^2 \cdot \frac{1-\eps}{4\ln n}\right).
\end{align}
We first show that the summand involving $\ell$ is increasing. 
Certainly the multiplicative factor of $\ell$ is increasing, so
let us define $x_\ell:=\exp\left(\ell -(1-\eps_0)(\ell-1)^2 (1-\eps)\frac{1}{4\ln n}\right)$. Then for $\ell \le k_0=2\ln n$ we have
$$
\frac{x_{\ell+1}}{x_\ell} = \exp \left(1-(1-\eps_0)(1-\eps)(2\ell-1)\frac{1}{4\ln n}\right)
\ge \exp \left(1-(1-\eps_0)(1-\eps)\right) \ge e^\eps>1.
$$
Therefore the sum over $\ell$ in~\eqref{eq:optconfigs} can be bounded from above by
replacing each summand by the summand with $\ell=k_0$, i.e.\ 
\begin{align}\label{eq:optconfigs:lsum}
\sum_{\ell=1}^{k_0}\ell\exp\left(\ell -(1-\eps_0)(\ell-1)^2 (1-\eps)\frac{1}{4\ln n}\right)\hspace{-2cm} & \nonumber \\
& \le k_0^2 \exp\left(k_0 -(1-\eps_0)(k_0-1)^2 (1-\eps)\frac{1}{2k_0}\right)\nonumber\\
& = k_0^2 \exp\left(k_0\left(1-\frac{(1-\eps_0)(1-\eps)}{2}\right)\right)
\exp\left(\frac{(2k_0-1)(1-\eps_0)(1-\eps)}{2k_0}\right)\nonumber\\
& \le k_0^2\exp\left(k_0\left(\frac{1 +\eps +\eps_0}{2}\right)\right)\cdot e.
\end{align}

On the other hand, considering the sum over $r$ in~\eqref{eq:optconfigs}, we approximate as follows.
Since $r\le k \le k_0$ and $k_0\ge e^2$, by \eqref{eq:stirling} we have $r!\le e \sqrt{r} (r/e)^r \leq k_0 (r/e)^r$, implying

\begin{align}\label{eq:lsum}
\sum_{r=k_0-k}^{k} r!r^3 \left(\frac{e(1-\eps)}{2\ln n}\right)^r 
\le k_0^4\sum_{r=k_0-k}^{k} \left(\frac{r}{e}\right)^r\left(\frac{e(1-\eps)}{k_0}\right)^r
&= k_0^4\sum_{r=k_0-k}^{k} \left(\frac{(1-\eps)r}{k_0}\right)^r\nonumber\\
&\le k_0^4 \left(\frac{1-\eps}{k_0}\right)^{k_0-k}\sum_{r=k_0-k}^{k} \frac{r^r}{k_0^{r-k_0+k}}\nonumber\\
&\le k_0^4 \left(\frac{1-\eps}{k_0}\right)^{k_0-k}
\sum_{r=k_0-k}^{k} \frac{k_0^{k_0}}{k^k} \left(\frac{k}{k_0}\right)^{r+k}.
\end{align}

Note that since $k_0-k \leq r \leq k$ we have
\begin{equation*}
1 \le \frac{k}{k}+\frac{r}{k}\left(1-\frac{k+r}{2k}\right) \le \frac{k+r}{k}-\frac{k_0-k}{k}\cdot \frac{k+r}{2k}.
\end{equation*}
Since $0 \leq \frac{k_0-k}{k} \leq 1,$ the Taylor expansion of $\ln(1+x)$ leads to 
\begin{equation*}
  (k+r) \ln{\left(1+\frac{k_0-k}{k}\right)} \geq (k+r) \frac{k_0-k}{k}-\frac{k+r}{2}\left(\frac{k_0-k}{k}\right)^2 = (k_0-k)\left(  \frac{k+r}{k}-\frac{k_0-k}{k}\cdot \frac{k+r}{2k} \right) \geq k_0-k.
\end{equation*}
Therefore, by \eqref{eq:stirling} we conclude that
$$\frac{k_0^{k_0}}{k^k}\left(\frac{k}{k_0}\right)^{r+k} \le \frac{k_0^{k_0}}{k^k} e^{k-k_0}\leq e\sqrt{k}\frac{k_0!}{k!} \le k_0 \frac{k_0!}{k!}.$$

Thus~\eqref{eq:lsum} gives
$$\sum_{r=k_0-k}^{k} r!r^3 \left(\frac{e(1-\eps)}{2\ln n}\right)^r\le k_0^6 \left(\frac{1-\eps}{k_0}\right)^{k_0-k}\frac{k_0!}{k!}.$$

Thus the sum over $k$ in~\eqref{eq:optconfigs} can be bounded by
\begin{align}\label{eq:optconfigs:ksum}
\sum_{k=k_0/2}^{k_0}k \cdot k! \left(\frac{1-\eps}{2\ln n}\right)^{k}
\sum_{r=k_0-k}^{k} r! r^3\left(\frac{e(1-\eps)}{2\ln n}\right)^r
& \le \sum_{k=k_0/2}^{k_0}k\cdot k! \left(\frac{1-\eps}{k_0}\right)^{k} k_0^6
\left(\frac{1-\eps}{k_0}\right)^{k_0-k} \frac{k_0!}{k!}\nonumber \\
& \le k_0^8 \left(\frac{1-\eps}{k_0}\right)^{k_0} k_0!
\end{align}

Substituting~\eqref{eq:optconfigs:lsum} and~\eqref{eq:optconfigs:ksum} into~\eqref{eq:optconfigs},
we obtain that the expected number of optimal configurations with parameters $(k,\ell,r)\in S$
is at most
\begin{align*}
\frac{ e^{582}nk_0}{1-\eps}
\left(k_0^8 \left(\frac{1-\eps}{k_0}\right)^{k_0} k_0!\right)
k_0^2\exp\left(k_0\left(\frac{1 +\eps +\eps_0}{2}\right)\right) \cdot e \hspace{-4cm} & \\
& \le nk_0^{12} \left(\left(\frac{1-\eps}{k_0}\right)^{k_0}k_0 \left(\frac{k_0}{e}\right)^{k_0}\right)
\exp\left(k_0\left(\frac{1+\eps+\eps_0}{2}\right)\right)\\
& \le nk_0^{13} \left(\frac{1-\eps}{e} \cdot \exp \left(\frac{1+\eps+\eps_0}{2}\right)\right)^{k_0}\\
& \le nk_0^{13} \left(e^{-1-\eps} \cdot \exp\left(\frac{1+\eps+\eps_0}{2}\right)\right)^{k_0}\\
& = n (2\ln n)^{13} \exp\left(\left(-\frac{1}{2}-\frac{\eps}{2}+\frac{\eps_0}{2}\right)2\ln n\right)\\
& = n (2\ln n)^{13} n^{\left(-1-\eps+\eps_0\right)}\\
& = (2\ln n)^{13} n^{-\eps+\eps_0}.
\end{align*}
Note that since we chose $\eps_0=\eps/3$ and $\eps$ is constant, this term tends to $0$.

Thus by Markov's inequality, with high probability there is no such percolating set, and therefore the double-graph does not percolate.

\section{Supercritical case: Proof of Theorem~\ref{thm:main}~\ref{thm:main:supercrit}}\label{sec:supercrit}

In this section we prove Theorem~\ref{thm:main}~\ref{thm:main:supercrit}.
Recall that we assume that $p_1 \ge p_2 \ge \frac{\ln n}{n}$ and the statement says that conditioned
on the individual graphs $G_1 \sim G(n,p_1)$ and $G_2 \sim G(n,p_2)$ being connected, with high probability the jigsaw process percolates on the double graph
$G(n,p_1,p_2)$.

We first note that for $p_1,p_2 \ge \frac{\ln n}{n}$, the probability of $G_1$ and $G_2$ being connected is bounded
below by a (non-zero) constant. Thus any event that holds with high probability also holds with high probability
in the probability space conditioned on $G_1$ and $G_2$ being connected. Therefore for simplicity in the
following arguments we will suppress this conditioning.

The proof consists of three stages in which we construct a nested sequence of percolating sets
$U_1 \subset U_2 \subset U_3 = V,$ growing in size as we reveal more edges.
In Section~\ref{sec:pastbottleneck} we define and analyse
a \emph{construction algorithm} (Algorithm \ref{algo:construct}) which constructs percolating sets,
and show in Lemma~\ref{lem:constrU1} that with high probability it constructs
at least one percolating set of reasonably large size. Subsequently, in Section~\ref{sec:sprinkling}
we show that this percolating set expands to cover almost all its red neighbours
(Lemma~\ref{lem:furext}). Finally we use a sprinkling argument
to extend this percolating set until it eventually covers all vertices and so prove the supercritical case.

We will reveal the red graph with two rounds of exposure.
To this end, set 
\begin{align*}
p_1^{(1)} &:= \left(1-\frac{\eps}{2}\right)p_1, \text{ and } p_1^{(2)} := \frac{\eps}{2} p_1
\end{align*}
and 
\begin{align*}
p_2^{(1)} &:= p_2, \text{ and } p_2^{(2)} := 0.
\end{align*}
Consider the double graphs $G(n,p_1^{(i)},p_2^{(i)})$ for $i=1,2$. For $i=1, 2$, we denote by $N_1^{(i)} (U)$
the neighbourhood of $U$ within $G(n,p_1^{(i)})$.

Let us note that $G(n,p_1^{(1)},p_2^{(1)})\cup G(n,p_1^{(2)},p_2^{(2)}) \sim G(n,p_1^*,p_2)$, where
$$
p_1^* = 1-(1-p_1^{(1)})(1-p_1^{(2)}) \le p_1^{(1)}+p_1^{(2)} = p_1.
$$
Since percolation is a monotone increasing property the probability that $G(n,p_1,p_2)$ percolates is at least the probability that $G(n,p_1^{(1)},p_2^{(1)})\cup G(n,p_1^{(2)},p_2^{(2)})$ percolates.

\subsection{Getting past the bottleneck} \label{sec:pastbottleneck}
We set
\begin{align*}
\omega = \omega_n :=\ln \ln n, \qquad  k_1:=\frac{1}{\omega p_1^{(1)}} \qquad\text{and} \qquad \delta := \frac{\eps}{20}.
\end{align*}
Our aim is to construct a percolating set of size $k_1$ by means of an algorithm. Refining an algorithm of Bollob\'as, Riordan, Slivken and Smith~\cite{BollobasRiordanSlivkenSmith17}, we grow a percolating set $X_t$ by adding vertices
in each step $t$.

We first describe the algorithm informally, suppressing the index $t$ for simplicity.
We begin with $X$ being a single vertex. 
At the start of step $t$, the set $X$ will consist of vertices $x_1,\ldots,x_{s}$ which form a percolating set,
with $s\ge t$.
In addition, the set $R$ consists of vertices which are adjacent to at least one of
$x_1,\ldots,x_{t-1}$ in red, but not in blue.

In step $t$, we will reveal the red neighbours $Q$ of $x_t$ (outside of $X\cup R$).
We will also reveal any blue edges between $Q$ and $x_1,\ldots,x_t$ -- vertices
incident to such a blue edge will be added to $X$, while the remaining vertices of $Q$
are added to $R$.
We also reveal any blue edges between
$R$ and $x_t$, and vertices incident to such an edge will be moved from $R$ to $X$.

There are two main differences between our algorithm and the algorithm described by
Bollob\'as, Riordan, Slivken and Smith~\cite{BollobasRiordanSlivkenSmith17}:
first, in their algorithm, they only add one vertex to $X$ in each step;
and second, they did not keep track of the set $R$ and reveal blue edges between
$x_t$ and $R$, but simply discarded it along with any other vertices that could have been added to $X$.
In order to prove the sharper version of the theorem with the exact threshold we require this more detailed algorithm,
and significantly more precise analysis until it passes the bottleneck.

We will run the algorithm several times. Each attempt is called a \emph{round}, indexed by $\ell$.
At the end of each round, we will discard the percolating set generated in the round -- this ensures
independence between rounds. In order to make the analysis of each round identical, we will artificially exclude some vertices
from each round to ensure that we always have the same number of vertices available.

\vspace{0.2cm}
\begin{algorithm}[H]\label{algo:construct}
	\SetAlgoLined
	\textbf{Input:} Double graph $(V,E_1,E_2)$.\\
	Set $\ell = 1$ and $V_1'=V.$ \\
	\While{$|V_\ell'|\ge n-n^{1-\delta}$}{
		Fix an arbitrary set $V_\ell \subset V_\ell'$ of size $\sizevl$. \\
		Pick an arbitrary vertex $x_1=x_1(\ell)\in V_\ell$ and set $X_0=X_0(\ell):=\{x_1\}$ and $s_0=s_0(\ell):=1$. \\
		Also set $R_0=R_0(\ell):=\emptyset$ and $t=1$.\\
		\While{$t\le s_{t-1}< k_1$}{
			Set $Q_t=Q_t(\ell)=N_1^{(1)}(x_t)\cap \left(V_\ell\setminus (X_{t-1}\cup R_{t-1})\right)$.\\
			Set $B_t=B_t(\ell)=N_2(x_1,\ldots,x_t) \cap Q_t$.\\
			Set $C_t=C_t(\ell)=N_2(x_t) \cap R_t$. \\
			Set $X_{t}=X_{t}(\ell)=X_{t-1}\cup B_t \cup C_t$ and $s_t=s_t(\ell):=|X_t|$.\\
			Set $R_{t}=R_{t}(\ell)=(R_{t-1}\cup Q_t)\setminus (B_t \cup C_t)$.\\
			Proceed to step $t+1$.
		}
		Set $T= T(\ell)=t-1$.\\
		Set $V_{\ell+1}'=V_{\ell}'\setminus X_{T(\ell)}$ and proceed
		to round $\ell+1$.
	}
	Set $L=\ell-1$.\\
	\textbf{Output:} $X_{T(1)},X_{T(2)},\ldots,X_{T(\ell-1)}$.
	\caption{{\sc The Construction Algorithm}}
\end{algorithm}
\vspace{0.2cm}

Note that $T(\ell)$ is the number of steps in round $\ell$, while $L$ is the number of rounds run by the construction algorithm.

We will apply the construction algorithm to $G(n,p_1^{(1)},p_2)$ and reveal edges only as they are required by the algorithm.
The following lemma shows that the construction algorithm is well-defined
and builds a percolating set.

\begin{lem}\label{lem:algowelldef}
	Algorithm~\ref{algo:construct} satisfies the following conditions:
	\begin{enumerate}
		\item \label{algcond:terminates} The algorithm terminates after a finite number of steps;
		\item \label{algcond:idptrounds} The rounds of the algorithm are mutually independent;
		\item \label{algcond:percolating} The set $X_{T(\ell)}$ forms a percolating set for $1 \le \ell \leq L.$
	\end{enumerate}
\end{lem}

\begin{proof}
	\ref{algcond:terminates}:
	For each round $\ell$ of the algorithm we perform $T(\ell)$ steps and $|X_{T(\ell)}|\ge T(\ell)$
	vertices are discarded, so the algorithm terminates after at most $n^{1-\delta}$ steps.
	
	\ref{algcond:idptrounds}:
	Within a round, any edge is revealed at most once and every queried pair is incident to $X_{T(\ell)}.$
	When the algorithm proceeds to the next round, it
	removes all the vertices of $X_{T(\ell)}$ from the vertex pool $V_\ell'$. Thus the algorithm queries every edge at most once.
	
	\ref{algcond:percolating}:
	Assume that $X_{t-1}$ forms a percolating set. Every element in $B_t$ has a red edge to $x_t$ and a blue edge into $\{ x_1, \dots, x_{t}\} \subset X_{t-1}$. Similarly the elements of $C_t$ have a blue edge to $x_t$ and a red edge into $\{ x_1, \dots, x_{t-1}\}$. 
	Consequently, $X_t$ also forms a percolating set and the assertion follows by induction over $t.$  
\end{proof}

The heart of the supercritical case is the following result.

\begin{lem}\label{lem:constrU1}
	Running the construction algorithm on $G(n,p_1^{(1)},p_2),$ with high probability there is a round $\ell$ such that
	$X_{T(\ell)}(\ell)$ has size at least $k_1$
	and $\left|R_{T(\ell)}{(\ell)}\right|\ge T(\ell)np_1^{(1)}/2$. 
\end{lem}  

The proof of Lemma \ref{lem:constrU1} will be given in Section~\ref{subsec:proofLemmaU1}. As preparation, we first 
approximate the sizes of various sets in the construction algorithm.

\subsubsection{Poisson approximation}
By Lemma~\ref{lem:algowelldef} the rounds of the construction algorithm are independent. Thus, the following results
hold uniformly for all $1\leq \ell \leq L$ and we will therefore lighten the notation by dropping $\ell$.
Moreover, we use the notation $a = b \pm c$ to mean that $b-c \leq a \leq b+c$,
and similarly $a = (b \pm c)d$ to mean $(b-c)d \leq a \leq (b+c)d$.

We aim to simplify the analysis of the algorithm by approximating the sizes of the various sets constructed.
In particular, our main aim is Lemma~\ref{lem:Xtincrements}, in which we approximate the distribution of the
number of vertices added to the percolating set in each step. In order to achieve this, we first need to know that various
other sets are about as large as we expect. 

\begin{definition}
	Set $\error := \frac{\varepsilon}{10}$ and define the events
	\begin{align*}
	\cQ_t & =\left\{|Q_t| = \left(1\pm \frac\error 2\right)np_1^{(1)}\right\},\\
	\cB_t & =\left\{ |B_t| < \frac \error 4np_1^{(1)}\right\},\\
	\cC_t & =\left\{ |C_t| < \frac \error 4np_1^{(1)}\right\},\\
	\cR_t & =\left\{ |R_t| = (1\pm \error)tnp_1^{(1)}\right\},\\
	\cH & =\bigcap_{t\leq T} \cH_t=\bigcap_{t\leq T} \cQ_t \cap \cB_t\cap \cC_t \cap \cR_t.
	\end{align*}
\end{definition}
The events $\cQ_t$ and $\cR_t$ state that $|Q_t|$ and $|R_t|$ are concentrated around their means.
Conditioned on $\cQ_t$ and $\cR_t$,
the expected sizes of $B_t$ and $C_t$ are about $tnp_1^{(1)}p_2$ and $(t-1)np_1^{(1)}p_2,$ respectively.
Thus, observing that $tp_2 \le k_1 p_1 = o(1)$, the events
$\cB_t$ and $\cC_t$ only require the corresponding random variables to be below a very crude upper bound.
As a preliminary, we show that these events are very likely to hold in every round of the algorithm.

\begin{lem}\label{lem:setsizes}
	During one round of the construction algorithm on $G(n,p_1^{(1)},p_2),$ the event $\cH$ holds with probability at least $1-\exp \left(-\Omega \left( n^{1/3} \right) \right)$.
\end{lem}

\begin{proof}
	We have
	\begin{equation*}
	\pr [\cH] = \prod_{t=1}^T \pr [\cH_t \mid \cH_1, \dots, \cH_{t-1} ]= \prod_{t=1}^T \pr [\cQ_t,\cB_t, \cC_t,\cR_t \mid \cH_1, \dots, \cH_{t-1}].
	\end{equation*}
	
	We will give a uniform lower bound for each of these terms. Recalling the definitions of
	$R_t,Q_t,B_t,C_t$ from the construction algorithm,
	conditional on $\cQ_t,\cB_t, \cC_t$ and $\cR_{t-1}\subset \cH_{t-1}$ we have
	$$|R_t|=|R_{t-1}|+|Q_t|-|B_t|-|C_t|=(1\pm \error)(t-1)np_1^{(1)}+\left(1\pm \frac{\error}{2}\right)np_1^{(1)}\pm \frac{\error}{4} np_1^{(1)} \pm \frac{\error}{4} np_1^{(1)}=(1\pm \error)tnp_1^{(1)},$$
	i.e.\ $\cR_t$ holds deterministically, implying
	\begin{align}
	\pr [\cQ_t,\cB_t, \cC_t,\cR_t \mid \cH_1, \dots, \cH_{t-1}]&=\pr [\cQ_t,\cB_t, \cC_t \mid \cH_1, \dots, \cH_{t-1}]\nonumber\\
	&=\pr [\cQ_t \mid \cH_1, \dots, \cH_{t-1}] \pr [\cB_t \mid \cQ_t, \cH_1, \dots, \cH_{t-1}] \pr [ \cC_t \mid \cQ_t, \cB_t, \cH_1, \dots, \cH_{t-1}].\label{eq:eventsprobprod}
	\end{align}
	
	Our goal is to show that each of these terms has probability $1-\exp\left(-\Omega(n^{1/3})\right)$.
	We will repeatedly use the fact that
	\begin{equation*}
	np_1^{(1)} = \Omega(np_1) \stackrel{\eqref{pre:p1lower}}{=} \Omega(n^{1/3}).
	\end{equation*}
	First note that
	$$|Q_t| \sim \Bi \left( \sizevl - |X_{t-1}| - |R_{t-1}| , p_1^{(1)} \right).$$
	Since $|X_{t-1}|\le k_1=o(n)$, and conditional on $\cH_{t-1}$, we have $|R_{t-1}|=(1\pm \error)(t-1)np_1^{(1)}=O(k_1 n p_1^{(1)})=o(n)$, thus
	$$\mathbb{E}[|Q_t|]=(1+o(1))np_1^{(1)}=\Omega(n^{1/3}).$$
	Together with the Chernoff bound (Lemma~\ref{lem:chernoff}), this implies
	$$\pr [\overline{\cQ}_t \mid \cH_1, \dots, \cH_{t-1}]\le \pr\left[ ||Q_t|-\mathbb{E}[|Q_t|]|\ge \frac{\error}{4}np_1^{(1)}\right]\le \exp\left(-\Omega \left(np_1^{(1)}\right)\right)=\exp\left(-\Omega(n^{1/3})\right).$$
	
	Next we consider $\cB_t$. Clearly
	$$|B_t| \sim \Bi \left( |Q_t| , 1-(1-p_2)^t \right),$$
	therefore conditional on $\cQ_t$ we have
	$$\mathbb{E}[|B_t|]=O\left(|Q_t|p_2t\right)=O\left(n p_1^{(1)}p_2 k_1\right)=o(np_2)=o(np_1^{(1)})$$
	and thus Lemma~\ref{lem:chernoff} implies that
	$$\pr [\overline{\cB}_t \mid \cQ_t, \cH_1, \dots, \cH_{t-1}]\le \exp\left(-\Omega \left( np_1^{(1)}\right)\right) =\exp\left(-\Omega(n^{1/3})\right).$$
	
	Finally 
	$$|C_t| \sim \Bi \left( |R_{t-1}| , p_2 \right)$$
	and conditional on $\cR_{t-1}\subset \cH_{t-1}$ we have
	$$\mathbb{E}[|C_t|]=O\left(|R_t|p_2\right)=O\left(k_1 n p_1^{(1)}p_2\right)=o(np_2)=o(np_1^{(1)}),$$
	and again Lemma~\ref{lem:chernoff} implies that
	$$\pr [\overline{\cC}_t \mid \cQ_t, \cB_t, \cH_1, \dots, \cH_{t-1}] \le \exp\left(-\Omega \left( np_1^{(1)}\right)\right)
	=\exp\left(-\Omega(n^{1/3})\right).$$
	Taking a union bound and substituting into~\eqref{eq:eventsprobprod}, we have
	$$
	1-\pr [\cQ_t,\cB_t, \cC_t,\cR_t \mid \cH_1, \dots, \cH_{t-1}] \le 3\exp\left(-\Omega(n^{1/3})\right).
	$$
	The statement follows by applying the union bound over all steps in the round of the algorithm, of which there are at most $k_1$, and observing that
	$$
	3k_1 \exp\left(-\Omega(n^{1/3})\right)= \exp\left(-\Omega(n^{1/3})\right).
	$$
\end{proof}

It will be convenient in future analysis to condition on $\cH$. Lemma \ref{lem:setsizes} tells us that this is reasonable.
In order to compare binomials with Poisson random variables, we need the following notation:
for a non-negative integer-valued random variable $X$ and $r\in \NN$ let $X_{\le r}$ be the cutoff transform of $X,$ i.e. the random variable with 
\begin{align*}
\Pr [X_{\le r} = t] = \begin{cases} \frac{\Pr [X=t]}{\Pr [X\le r]}, & \text{ for } t\leq r \\
0 & \text{ else.}
\end{cases}
\end{align*} 

	The following claim shows how binomials dominate Poisson variables with suitable cutoff.
	
	\begin{claim}\label{claim:bindompoi}
		Let $X\sim \Bi ( N, p) $ and $Y \sim \Po_{\le r} ((1-\theta)Np)$, with $N>0$ and $r/N < \theta<1$.
		Then $ X \succ Y.$
	\end{claim}
	
	\begin{proof}
		Clearly for every $i>r$ we have $\pr[X\geq i]\ge \pr[Y\ge i]=0$. For $0\le i<r$ we have
		\begin{align*}
		\frac{\Pr [Y=i]}{\Pr [Y=i+1]} \frac{\Pr [X=i+1]}{\Pr [X = i]} &= \frac{i+1}{Np(1-\theta)}\frac{(N-i)p}{(i+1)(1-p)}= \frac{1-\frac{i}{N}}{(1-\theta)(1-p)} >1,
		\end{align*}
		implying
		$$\frac{\pr[X=i+1]}{\pr[X=i]}\ge \frac{\pr[Y=i+1]}{\pr[Y=i]}.$$

		Now suppose for a contradiction that for some $\ell \le r$ we have $\Pr[Y\ge \ell] > \Pr[X\ge \ell]$.
		It follows that $\Pr[Y= \ell] > \Pr[X= \ell]$, and hence it also follows that
		$\Pr[Y=i] > \Pr[X=i]$ for all $0 \le i \le \ell$.
		But then we have
		$$
		1=\Pr[Y\ge 0] = \sum_{i=1}^{\ell-1} \Pr[Y=i] + \Pr[Y\ge \ell] > \sum_{i=1}^{\ell-1} \Pr[X=i] + \Pr[X\ge \ell] = \Pr[X\ge 0]=1
		$$
		which is a contradiction.
	\end{proof}

It is well-known that the sum of Poisson variables is also Poisson, but we will need a similar result
for Poisson variables with a cutoff.

	\begin{claim}\label{claim:sumpo}
		For every $r\ge 0$ we have
		$$\Po_{\le r}(\lambda)+\Po_{\le r}(\mu) \succ \Po_{\le r}((\lambda+\mu)).$$
	\end{claim}
	
	\begin{proof}
		Note that for $i\le r$ we have
		$$\Pr[\Po_{\le r}(\lambda)+\Po_{\le r}(\mu)=i]=\frac{\Pr[\Po(\lambda+\mu)=i]}{\pr [\Po(\lambda)\le r] \pr [\Po(\mu)\le r]}
		\le \frac{\Pr[\Po(\lambda+\mu)=i]}{\pr [\Po(\lambda+\mu)\le r]} = \Pr[\Po_{\le r}(\lambda+\mu)=i]
		$$
		as required.
	\end{proof}

Our cutoff point will be at $\rho:=\omega^{-1}np_1^{(1)}$.

\begin{lem}\label{lem:Xtincrements}
	For any round and any step $t\leq T$ of the construction algorithm on $G(n,p_1^{(1)},p_2),$ conditional on $\cH$ we have
	$$|X_{t}|-|X_{t-1}| \succ \Po_{\le \rho} \left(\left(1+\frac{\eps}{5}\right)\frac{2t-1}{4\ln n}\right).$$
\end{lem}

\begin{proof}
	Conditional on $\cH$, the increment $|X_{t}|-|X_{t-1}| = |B_t| + |C_t|$ dominates  the sum of two independent binomials
	$ B_t^- \sim \Bi\left((1- \frac \error 2)np_1^{(1)},(1-\frac{\error}{2})tp_2\right)$
	and $C_t^- \sim \Bi \left((1- \error)(t-1)np_1^{(1)},p_2\right).$

	Set $\theta = \error$ and for $t>1$ we have
$$\frac{\rho}{(1-\error)(t-1)np_1^{(1)}}\le  \frac{\omega^{-1}}{1-\error} =  o(1) < \error$$ 
and
$$\frac{\rho}{(1-\error/2)np_1^{(1)}} = \frac{\omega^{-1}}{1-\error/2} =  o(1) < \error.$$
	Hence Claims~\ref{claim:bindompoi}, together with $C_1^-=\Po_{\le \rho}(0)=0$, and Claim~\ref{claim:sumpo} yield
	\begin{align*}
	B_t^- + C_t^- & \succ \Po_{\le \rho} \left((1-\error)^2tnp_1^{(1)}p_2\right) + \Po_{\le \rho} \left((1-\error)^2(t-1)np_1^{(1)}p_2\right)\\
	& \succ \Po_{\le \rho} \left((1-\error)^2(2t-1)np_1^{(1)}p_2\right).
	\end{align*}
	
	Now Lemma~\ref{lem:Xtincrements} follows immediately, since
	$$(1-\error)^2 np_1^{(1)}p_2 = (1-\error)^2 \left(1- \frac{\varepsilon}{2}\right) \frac{1+ \varepsilon} {4\ln n} \geq   \frac{1+\varepsilon/5}{4\ln n}.$$
\end{proof}

\subsubsection{Two-stage analysis}
We break the proof of Lemma~\ref{lem:constrU1} into two stages. To this end,  for $k \in \NN$ define
\begin{align*}
\cE_k := \lbrace |X_{T}| \geq k+1   \rbrace ,
\end{align*}
i.e.\ $\cE_k$ is the event that the current round of the construction algorithm
finds a percolating set of size at least $k+1$, or equivalently that it survives for at least $k$ steps.
First, we show that percolating sets of size $k_0:= 2 \ln n$ (just above the bottleneck) are not too unlikely.
Recall that $\delta = \frac{\eps}{20}$.

\begin{lem}\label{lem:tobottleneck}
	We have $ \Pr[\cE_{k_0}\mid \cH] \geq n^{-1+2\delta}. $
\end{lem}

\begin{proof}
	Let $Z_1,Z_2,\ldots$ be a family of independent random variables with distribution
	$$Z_t \sim \Po_{\le \rho} \left(\left(1+\frac{\eps}{5}\right)\frac{2t-1}{4\ln n}\right).$$
	A sufficient condition for the construction algorithm to survive $k$ steps in a round is that the sum of increments $|X_t|-1= \sum_{1\le s \leq t} (|X_s|-|X_{s-1}|)$ never drops below $t$ for $1\le t \le k$. Due to Lemma \ref{lem:Xtincrements} it  holds that 
	\begin{align}\label{eq:probek}
	\Pr [\mathcal{E}_k \mid \cH] \geq \Pr\left[ \bigwedge_{t=1}^k \sum_{s=1}^t Z_s \geq t \right].
	\end{align}
	We write $\ii$ for a vector $(i_1, \dots, i_k) \in [k]^k$ and set 
	\begin{align*}
	\mathcal{A}_k:=\left\{\ii \in [k]^k: \sum_{t=1}^k i_t=k \right\} \qquad \text{and} \qquad \mathcal{A}_k^*:=\left\{\ii \in \mathcal{A}_k: \bigwedge_{t=1}^k \sum_{s=1}^t i_s \geq t\right\}.
	\end{align*}
	Consequently it holds that $ \Pr\left[ \bigwedge_{t=1}^k \sum_{s=1}^t Z_s \geq t \right] \ge
	\sum_{\ii \in \mathcal{A}_k^*} \prod_{t=1}^k \Pr [Z_t=i_t]$.
	The additional, seemingly arbitrary restriction that the entries $i_t$ sum up to $k$ will turn out to be very useful for the following analysis.
	Set
	$$S_{k_0}:=\sum_{\ii \in \mathcal{A}_{k_0}^*}\prod_{t=1}^{k_0} \frac{(2t-1)^{i_t}}{i_t!}.$$
	Since $k_0\le \rho$ we have
	\begin{align}
	\Pr\left[ \bigwedge_{t=1}^{k_0} \sum_{s=1}^t Z_s \geq t \right]&\geq \sum_{\ii \in \mathcal{A}_{k_0}^*} \prod_{t=1}^{k_0} \exp \left(-\left(1+\frac \eps 5\right)\frac {2t-1}{4\ln{n}}\right)\frac{\left( \left(1+\frac \eps 5\right)\frac{2t-1}{4\ln{n}}\right)^{i_t}}{i_t!}
	\frac{1}{\pr\left[\Po\left(\left(1+\frac{\eps}{5}\right)\frac{2t-1}{4\ln n}\le \rho \right)\right]} \nonumber \\
	&\ge \exp\left(-\frac{1+\frac \eps 5}{4\ln{n}}k_0^2\right) S_{k_0} \left(\frac{1+\frac \eps 5}{4\ln{n}}\right)^{k_0} \cdot 1.\label{eq:probek0}
	\end{align}
	Moreover, defining $m:=k_0^{2/3}$ and 
	\begin{align*}
	\tilde{\mathcal{A}}_{k_0}:= \left\{\ii \in \mathcal{A}_k: \bigwedge_{t=1}^k \sum_{s=1}^t i_s < t +m\right\},
	\end{align*}
	we observe that for $\ii \in \cA_{k_0}^* \cap \tilde{\cA}_{k_0} $ the product $\prod_{t=1}^{k_0}(2t-1)^{i_t}$ is bounded below by $ 1^{m+1}\cdot 3\cdot 5\cdots (2(k_0-m)-1) =  (2k_0-2m-1)!!.$  Hence 
	\begin{align}\label{eq:sk0help}
	S_{k_0} & \ge (2k_0-2m-1)!! \sum_{\ii \in \cA_{k_0}^* \cap \tilde{\cA}_{k_0}} \prod_{t=1}^{k_0}\frac{1}{i_t!}
	= \frac{(2k_0-2m)!}{2^{k_0-m}(k_0-m)!}\cdot \frac{k_0^{k_0}}{k_0!}
	\left(\frac{1}{k_0^{k_0}}\sum_{\ii \in \cA_{k_0}^* \cap \tilde{\cA}_{k_0}} \binom{k_0}{i_1, \dots, i_{k_0}} \right).
	\end{align}

	Let $\UU = (U_1, \dots, U_{k_0})\in \cA_{k_0}$ be a random vector created by
	independently assigning $k_0$ labelled balls into $k_0$ labelled bins,
	where $U_s$ denotes the number of balls in bin $s$. Then the term in brackets
	describes the probability that $\UU \in \cA_{k_0}^* \cap \tilde{\cA}_{k_0}.$
	Clearly 
	$$\Pr [\UU \in \cA_{k_0}^* \cap \tilde{\cA}_{k_0}] \geq \Pr [\UU \in \cA_{k_0}^*]- \Pr[\UU \notin \tilde{\cA}_{k_0} ]$$
	and thus we need a lower bound on $\Pr [\UU \in \cA_{k_0}^*]$ and an upper bound on $\Pr[\UU \notin \tilde{\cA}_{k_0} ]$.
	
	First, let $t^*$ be the largest index such that
	$$\sum_{s=1}^{t^*}U_{s} - t^* = z :=\min_{1\le t \le k_0} \sum_{s=1}^{t}U_{s} - t.$$
	We claim that $\left( U_{t^*+1}, \dots, U_{k_0}, U_1, \dots U_{t^*}\right) \in \cA_{k_0}^*$ implying $\Pr [\UU \in \cA_{k_0}^*]\geq 1/k_0$.
	For observe that certainly $U_{t^*+1}+ \ldots + U_i \ge i-t^*$ for $t^*+1\le i \le k_0$
	by the definition of $t^*$. Furthermore
	$$U_{t^*+1}+ \ldots + U_{k_0} = k_0-\sum_{s=1}^{t^*}U_s=k_0-t^*-z,$$
	and so for $1 \le i \le t^*$ we have
	$$U_{t^*+1}+ \ldots + U_{k_0} + U_1 + \ldots +U_i \ge k_0-t^* -z + (i+z) \ge k_0-t^* +i$$
	by the definition of $z$, as required.
	Secondly, since $\sum_{s=1}^t U_s \sim \Bi (k_0 , t/k_0),$ by Lemma~\ref{lem:chernoff} and a union bound, we obtain
	\begin{align*}
	\Pr[\UU \notin \tilde{\cA}_{k_0} ] = \Pr \left[ \bigcup_{t=1}^{k_0} \left\{ \sum_{s=1}^t U_s > t + m \right \} \right] \leq \sum_{t=1}^{k_0} \Pr \left[ \sum_{s=1}^t U_s > t + m\right] \leq k_0 \exp \left( - \frac{m^2}{3k_0}\right) \leq  k_0\exp (- k_0^{1/3} /3).
	\end{align*}
	Consequently, 
	\begin{align*}
	\frac{1}{k_0^{k_0}}\sum_{\ii \in \tilde{\mathcal{A}}_{k_0}} \binom{k_0}{i_1, \dots, i_{k_0}} = \Pr [\UU \in \cA_{k_0}^* \cap \tilde{\cA}_{k_0}] \geq \frac{1}{k_0}-k_0 \exp(-k_0^{1/3}) \geq k_0^{-2}.
	\end{align*}
	Hence, \eqref{eq:sk0help} and~\eqref{eq:stirling} yield 
	\begin{align}\label{eq:sk0final}
	S_{k_0} \ge \frac{(2(k_0-m))^{2(k_0-m)}}{2^{k_0-m}(k_0-m)^{k_0-m}}\frac{e^{m-1}}{k_0(k_0-m)} \frac{1}{k_0^2}
	& \geq \left(2k_0-2m\right)^{k_0-m}\frac{e^ {m-1}}{k_0^4}\nonumber \\
	& = \exp (k_0 \ln (2k_0) - o(k_0)).
	\end{align}
	Combining \eqref{eq:probek}, \eqref{eq:probek0} and \eqref{eq:sk0final}
	gives us
	\begin{align*}
	\Pr[\cE_{k_0} \mid \cH]
	&\ge 
	\exp\left( k_0 \left(-\frac{1+\frac \eps 5 }{4\ln{n}}k_0 + \ln (2k_0) - o(1)\right)\right)
	\left(\frac{1+\frac \eps 5}{4\ln{n}}\right)^{k_0} \\
	&= \exp\left(2\ln n\left(-\frac 12 \left(1+\frac \eps 5\right) +\ln \left( 1+\frac \eps 5 \right) - o(1) \right) \right)\\
	&= n^{-1}\exp\left(2\ln n\left(-\frac \eps {10} +\ln \left( 1+\frac \eps 5 \right) - o(1) \right) \right)\\
	& \ge n^{-1+2\delta},
	\end{align*}
	where the last line holds since
	$-\frac{\eps}{10}+ \ln\left( 1+\frac \eps 5  \right) - o(1) \geq -\frac{\eps }{10} + \frac{\eps}{5} -\frac{1}{2}\left(\frac{\eps}{5}\right)^2 -o(1)
	> \frac{\eps }{20}=\delta$ for sufficiently small $\eps$. 
\end{proof}

Lemma~\ref{lem:tobottleneck} gave a lower bound on the probability of constructing
a percolating set of size $k_0$ in one round of the construction algorithm.
Subsequently there is a small but constant probability of growing a percolating set of size $k_1$
from a percolating set of size of $k_0$.

\begin{lem}\label{lem:frombottleneck}
	If we run a round of the construction algorithm in $G(n,p_1^{(1)},p_2),$ then we have
	$\Pr \left[  \cE_{k_1} | \cE_{k_0},\cH  \right] = \Theta (\eps) .$
\end{lem}

\begin{proof}
	We view the percolating set constructed in Algorithm~\ref{algo:construct} as a graph branching process,
	in which the vertex $x_t$ gives birth to the vertices in $B_t\cup C_t$.
	(Note that in fact while they are certainly each adjacent to $x_t$ in one colour, they may not
	be adjacent in both, so we are constructing an auxiliary graph.)
	Lemma~\ref{lem:Xtincrements} yields that, conditional on $\cH$, the branching process
	up to termination of the round, i.e.\ until it dies out or reaches size $k_1$,
	dominates a branching process with offspring distribution $\Po_{\le \rho}(1+\eps/5)$.
	Since the expected number of offspring is $1+\Theta(\eps)$,
	this branching process survives forever with probability $\Theta(\eps)$.
	Therefore conditioned on the percolating set constructed by Algorithm~\ref{algo:construct} reaching size $k_0$,
	with probability $\Theta(\eps)$ it will also reach size $k_1$. 
\end{proof}

\subsubsection{Proof of Lemma~\ref{lem:constrU1}}\label{subsec:proofLemmaU1}
With regard to the first statement of Lemma~\ref{lem:constrU1}, i.e.\ that whp there exists a round $\ell$ in which $|X_{T(\ell)}(\ell)|\ge k_1$,
define the event
$\cD=\bigcap_{\ell \le L}\overline{\cE_{k_1-1}(\ell)}$. Then the assertion is simply $\pr\left[\cD \right]=o(1)$.
Let $\cH^*$ denote the event that
$\cH(\ell)$ holds for every $1\le \ell \le L$. By Lemma~\ref{lem:setsizes}, $\pr\left[\cD \right]=o(1)$ follows from $\pr[\cD \mid \cH^*]=o(1)$.
Observe that if $\cD$ holds, then we discarded at most $k_1$ vertices in each round of the algorithm.
Now let $L_0$ be the number of rounds in which $\cE_{k_0}(\ell)$ does \emph{not} hold, and $L_1$ be the
number of rounds in which $\cE_{k_0}(\ell)$ \emph{does} hold. Thus $L_0+L_1=L$.
Furthermore, if $\cD$ holds, then we have deleted at most $k_0L_0+k_1L_1$ vertices during the algorithm,
and therefore
$$k_0L_0+k_1L_1\ge n^{1-\delta}.$$
We show that this is very unlikely by observing that by Lemma~\ref{lem:frombottleneck},
$\pr\left[\cD \cond L_1,\cH^*\right]\leq (1-c\eps)^{L_1}$ for some constant $c>0$, while
by Lemma~\ref{lem:tobottleneck}, conditional on $\cH^*$, we have $L_1\succ \Bi(L,n^{-1+2\delta})$.

We analyse
$$
\sum_{\ell_0,\ell_1} \pr[L_0=\ell_0, L_1=\ell_1, \cD \mid \cH^*].
$$
We split into various cases. Firstly, if $L_1$ is large, then $\cD$ is very unlikely:
\begin{align*}
\sum_{\ell_0} \sum_{\ell_1\ge \ln n} \pr[L_0=\ell_0, L_1=\ell_1, \cD \mid \cH^*]
& \le \sum_{\ell_1\ge \ln n} \pr\left[\cD\cond L_1=\ell_1,\cH^*\right] \le  \sum_{\ell_1\ge \ln n} (1-c\eps)^{\ell_1} \le \frac{\exp(-c\eps \ln n)}{c\eps} = o(1).
\end{align*}
On the other hand, we show that it is very unlikely that $L_0$ is large, but $L_1$ is small:
\begin{align*}
\sum_{\ell_0\ge n^{1-3\delta/2}} \sum_{\ell_1< \ln n} \pr [L_0=\ell_0, L_1=\ell_1, \cD \mid \cH^*]
& \le \sum_{\ell_0\ge n^{1-3\delta/2}} \sum_{\ell_1< \ln n} \pr [ L_1=\ell_1 \mid L_0=\ell_0, \cH^*]\\
& \le \sum_{\ell_0\ge n^{1-3\delta/2}} \pr\left[\Bi(\ell_0,n^{-1+2\delta})\le \ln n\right]\\
& \le n \cdot \pr\left[\Bi(n^{1-3\delta/2},n^{-1+2\delta})\le \ln n\right].
\end{align*}
Using Lemma~\ref{lem:chernoff}, we obtain 
\begin{align*}
\pr\left[\Bi(n^{1-3\delta/2},n^{-1+2\delta})\le \ln n\right] &\leq \pr\left[\Bi(n^{1-3\delta/2},n^{-1+2\delta})\le \left( 1-\delta\right) n^{\delta /2} \right] \leq \exp \left[ - \frac{n^{\delta /2} \delta^2}{2} \right]
\end{align*}
and, consequently, 
\begin{align*}
\sum_{\ell_0\ge n^{1-3\delta/2}} \sum_{\ell_1\le \ln n} \pr[L_0=\ell_0, L_1=\ell_1, \cD \mid \cH^*] = o(1).
\end{align*} 
Finally, observe that if both $L_0$ and $L_1$ are small, but $\cD$ holds, then we cannot have terminated
the algorithm because we have not deleted enough vertices: if $\cD$ holds, $\ell_0< n^{1-3\delta/2}$ and $\ell_1< \ln n,$ then
\begin{align*}
n^{1-\delta} \le L_0k_0+L_1k_1 & \le n^{1-3\delta/2}2\ln n + \ln n \frac{1}{\omega p_1}\\
& = o(n^{1-\delta}) + o(\ln n \sqrt{n\ln n}) = o(n^{1-\delta}),
\end{align*}
which is clearly a contradiction. Thus we have $\Pr[\cD \mid \cH^*] = o(1)$.

Moreover, from Lemma~\ref{lem:setsizes} we obtain that conditional on $\cH^*$
\begin{align*}
\left|R_{T(\ell)}{(\ell)}\right| \geq \frac{1}{2} T(\ell) np_1^{(1)}.
\end{align*}
Finally, due to Lemma~\ref{lem:setsizes} this yields
\begin{align*}
\Pr \left[ \bigcup_{\ell \leq L} \cE_{k_1-1}(\ell) \cap  \left\{ \left|R_{T(\ell)}\right| \geq \frac{1}{2}n p_1^{(1)} T(\ell)\right\}  \right]
& \ge \Pr [\overline\cD \cap \cH^*]
\ge \Pr [\overline\cD \mid \cH^*]- \Pr [\bar{\cH^*}] = 1-o(1)
\end{align*}
as required.
\qed

\subsection{Final stages}\label{sec:sprinkling}

In the previous section we found a step $\ell$ such that $|X_{T(\ell)}|\ge k_1$ and $|R_{T(\ell)}|\ge T(\ell) n p_1^{(1)}/2$. 
Once again we drop the $\ell$ from our notation.

We now show that $X_T$ grows into a larger percolating set by examining the red neighbourhood of $X_T$. 
Note that until this point no edge between $x_{k_1}$ and $R_T$ has been revealed. In addition
any blue edge we have revealed so far is incident to a vertex of
$D_1:=(V\setminus V_L')$.

\begin{lem}\label{lem:furext}
	With high probability $G(n,p_1^{(1)},p_2^{(1)})$ contains a percolating set of size $n/(4\omega)$.
\end{lem}

\begin{proof}
	By Lemma~\ref{lem:constrU1}, whp we have $|X_T|\ge k_1$ and $|R_T|\ge T np_1^{(1)}/2$.
	Until this point we have only exposed the (partial) red neighbourhood of $\{x_1,\ldots,x_T\}$.
	Now we expose the red neighbourhood in $V_L'$ of $\{x_{T+1},\ldots,x_{k_1}\}$,
	and denote this red neighbourhood by $R'$. Clearly
	$$|R'|\sim \Bi\left(|V_L'|,1-\left(1-p_1^{(1)}\right)^{k_1-T}\right)$$
	and
	since $\mathbb{E}[|R'|] \ge (1-o(1))n(k_1-T)p_1^{(1)}\to \infty$,
	Lemma~\ref{lem:chernoff} implies that whp $|R'|\ge (k_1-T)np_1^{(1)}/2$.
	Therefore for $R=(R_T\cup R')\setminus D_1$ we have 
	$$|R|\ge k_1np_1^{(1)}/2-n^{1-\delta}-k_1 \ge n/(3\omega).$$
	
	Now $X_T$ forms a percolating set and every vertex in $R$ has a red neighbour in $X_T$,
	and therefore the (blue) component of $G_2[\{x_{k_1}\}\cup R]$ containing $x_{k_1}$
	can be added to the percolating set.
	Recall that no blue edges have been exposed in $\{x_{k_1}\}\cup R$, and therefore
	$G_2[\{x_{k_1}\}\cup R]\sim G(|R|+1,p_2)$.
	This graph has expected average degree $|R|p_2 \ge np_2/(3\omega)=\omega(1)$
	and therefore whp has a giant component covering all but $o(|R|)$ vertices, and
	in particular containing $x_{k_1}$, and the result follows.
\end{proof}

We can now complete the proof of the supercritical case

\begin{proof}[Proof of Theorem~\ref{thm:main}~\eqref{thm:main:supercrit}]
	Recall that since we assume that $p_2 \ge \frac{\ln n}{n}$,
	the probability that $G_2\sim G(n,p_2)$ is connected is at least a positive constant.
	Therefore any event that occurs with high probability also occurs with high probability in the probability space
	conditioned on $G_2$ being connected.
	\footnote{Note that this is the only point in the argument at which we need to
	condition on $G_2$ being connected.
	We also no longer need to assume that $G_1$ is connected since we assumed (wlog) that
	$p_1\ge p_2$, and it follows from~\eqref{pre:p1lower} that $G_1$ is connected whp.}

	In particular, let $U_2$ be the percolating set provided with high probability by Lemma~\ref{lem:furext}.
	For all $v\notin U_2,$ we have
	\begin{align*}
	\pr [ v \notin N^{(2)}_1(U_2) ] = \left( 1-\frac{\eps p_1}{2} \right)^{\frac{n}{4\omega}}\leq \exp \left( -  \frac{\eps n p_1 }{8\omega }\right) \le \exp\left(-n^{1/3}\right) = o(n^{-2}),
	\end{align*}
	where we have used~\eqref{pre:p1lower} and the fact that $\omega = \ln \ln n$ is subpolynomial.
	Hence a union bound over all at most $n$ vertices of $V\setminus U_2$ shows that whp all are in $ N^{(2)}_1(U_2)$.
	On the other hand, since the blue graph is connected by assumption,  it is easy to see that the jigsaw process will percolate.
\end{proof}

\section{Concluding Remarks} \label{sec:concluding}

\subsection{The critical window}
We have proved that Theorem~\ref{thm:main} for $\eps >0$ an arbitrarily small constant.
However, we note that for connectedness, of which jigsaw percolation may be considered
the double-graph analogue, a much stronger result is true. Namely the classical result of
\Erdos\ and \Renyi~\cite{ErdosRenyi59} implies that if
$p= \frac{\ln n + c_n}{n}$, then whp $G(n,p)$ is not connected if $c_n \to -\infty$
and whp $G(n,p)$ is connected if $c_n \to \infty$. In other words,
setting $p=(1- \eps)\frac{\ln n}{n}$ in the subcritical case, or $p=(1+\eps)\frac{\ln n}{n}$,
whp we have $G(n,p)$ being disconnected or connected respectively
provided that $\eps \gg (\ln{n})^{-1}$.

Similarly, it would be interesting to know for which $\eps = o(1)$ the statement of Theorem~\ref{thm:main}
is still true. With a little more care, our proof would show that
$\eps \gg (\ln{n})^{-1/4}$ is sufficient, but it seems likely that this is not best possible.

The key step required to understanding the critical window seems to be the number of minimal percolating sets on $k$ vertices. We provide upper and lower bounds on the asymptotics of this value, which differ by a factor of $e^{o(k)}$. More precise estimates on this value translate into sharper bounds on the threshold.

\subsection{Generalisations}
It would also be interesting to determine the exact threshold for the various
generalisations of Theorem~\ref{thm:BRSS}, including the analogous results
for multiple graphs~\cite{CG17} and for hypergraphs~\cite{BCKK17}. The latter would
be a particular challenge since the proof of the supercritical case in~\cite{BCKK17}
simply involved a reduction to the graph case, i.e. Theorem~\ref{thm:BRSS}.
Since Theorem~\ref{thm:main} is a strengthening of Theorem~\ref{thm:BRSS},
it also makes the hypergraph result stronger; however, the reduction step is not optimal,
and it seems likely that significant new ideas would be required.

\subsection{Other random graph models}
Real world graphs, in particular social networks, tend to have a power law degree distribution.
The binomial random graph does not have this property; however several other random graph models do,
for example the preferential attachment model (introduced in \cite{BA99} and rigorously defined in \cite{BRST01})
and random graphs on the hyperbolic plane (introduced in \cite{KPKVB10}).
The threshold for jigsaw percolation when the people graph is modelled by such a random graph
and for any random or deterministic choice of the puzzle graph
is still unknown. Indeed, apart from a brief one-directional result in~\cite{BrummittChatterjeeDeySivakoff15},
jigsaw percolation involving random graphs with a power-law degree distribution have not been studied.

\subsection{Speed of percolation}
One might also ask how many steps it takes for the jigsaw process to percolate
in the supercritical case, i.e.\ how often we have to construct an auxiliary graph
and merge the components in Algorithm~\ref{alg:jigsaw}. With a little care,
the arguments in this paper could be adapted to show that,
for $p_1p_2=\frac{1+\eps}{4n\ln n}$ where $\eps >0$ is constant,
whp at most $O(\ln n)$ steps are required, and indeed this can even be improved to
$(1+o(1))2\ln n$.
It would be interesting to know whether this upper bound is in fact tight whp.

\section{Acknowledgement}

The authors would like to thank Kathrin Skubch for helpful discussions in the early stages of this project.

\bibliographystyle{plain}

\bibliography{../../References}

\begin{thebibliography}{10}

\bibitem{MR1688958}
G.~E. Andrews, R.~Askey, and R.~Roy.
\newblock {\em Special functions}, volume~71 of {\em Encyclopedia of
  Mathematics and its Applications}.
\newblock Cambridge University Press, Cambridge, 1999.

\bibitem{BA99}
A.~Barab\'asi and R.~Albert.
\newblock Emergence of scaling in random networks.
\newblock {\em Science}, 286(5439):509--512, 1999.

\bibitem{BCKK17}
B.~Bollob\'as, O.~Cooley, M.~Kang, and C.~Koch.
\newblock Jigsaw percolation on random hypergraphs.
\newblock {\em J. Appl. Probab.}, 54(4):1261--1277, 2017.

\bibitem{BollobasRiordanSlivkenSmith17}
B.~Bollob\'as, O.~Riordan, E.~Slivken, and P.~Smith.
\newblock The threshold for jigsaw percolation on random graphs.
\newblock {\em Electron. J. Combin.}, 24(2):Paper 2.36, 14, 2017.

\bibitem{BRST01}
B.~Bollob\'as, O.~Riordan, J.~Spencer, and G.~Tusn\'ady.
\newblock The degree sequence of a scale-free random graph process.
\newblock {\em Random Structures Algorithms}, 18(3):279--290, 2001.

\bibitem{BollobasThomason85}
B.~Bollob{\'a}s and A.~Thomason.
\newblock {\em Random graphs of small order}, volume 118 of {\em North-Holland
  Math. Stud.}, pages 47--97.
\newblock North-Holland, Amsterdam, 1985.

\bibitem{BrummittChatterjeeDeySivakoff15}
C.~D. Brummitt, S.~Chatterjee, P.~S. Dey, and D.~Sivakoff.
\newblock Jigsaw percolation: {W}hat social networks can collaboratively solve
  a puzzle?
\newblock {\em Ann. Appl. Probab.}, 25(4):2013--2038, 2015.

\bibitem{BPPSH10}
S.~V. Buldyrev, R.~Parshani, G.~Paul, H.~E. Stanley, and S.~Havlin.
\newblock {Catastrophic cascade of failures in interdependent networks}.
\newblock {\em Nature}, 464(7291):1025--1028, April 2010.

\bibitem{CG17}
O.~Cooley and A.~Guti\'errez.
\newblock Multi-coloured jigsaw percolation on random graphs.
\newblock Submitted. arXiv:1712.00992.

\bibitem{ErdosRenyi59}
P.~Erd{\H{o}}s and A.~R{\'e}nyi.
\newblock On random graphs. {I}.
\newblock {\em Publ. Math. Debrecen}, 6:290--297, 1959.

\bibitem{GravnerSivakoff17}
J.~Gravner and D.~Sivakoff.
\newblock Nucleation scaling in jigsaw percolation.
\newblock {\em Ann. Appl. Probab.}, 27(1):395--438, 2017.

\bibitem{JansonLuczakRucinskiBook}
S.~Janson, T.~{\L}uczak, and A.~Ruci\'nski.
\newblock {\em Random graphs}.
\newblock Wiley-Interscience Series in Discrete Mathematics and Optimization.
  Wiley-Interscience, New York, 2000.

\bibitem{KPKVB10}
D.~Krioukov, F.~Papadopoulos, M.~Kitsak, A.~Vahdat, and M.~Bogu\~n\'a.
\newblock Hyperbolic geometry of complex networks.
\newblock {\em Phys. Rev. E (3)}, 82(3):036106, 18, 2010.

\bibitem{MR0069328}
H.~Robbins.
\newblock A remark on {S}tirling's formula.
\newblock {\em Amer. Math. Monthly}, 62:26--29, 1955.

\end{thebibliography}

\end{document}